\def\thetitle{Instrumental variables regression}

\def\theruntitle{gaussian comparison}

\def\theabstract{
	IV regression in the context of a re-sampling is considered in the work. Comparatively, the contribution in the development is a structural identification in the IV model. The work also contains a multiplier-bootstrap justification.  
}

\def\amsPrime{60F17}
\def\amsSec{60F25,60B12}

\def\thekeywords{Kolmogorov distance, Gaussian Comparison}

\def\authora{Andzhey Koziuk}
\def\authorb{Vladimir Spokoiny}
\def\runauthora{koziuk, a.}
\def\runauthorb{spokoiny, v.}
\def\addressa{
	Weierstrass Institute, \\
	Institute for Information Transmission Problems of RAS \\
	Mohrenstrasse 39 \\
	10117 Berlin, Germany
}
\def\addressb{
	Weierstrass Institute and \\ Humboldt University Berlin, \\ Moscow Institute of
	Physics and Technology \\
	Mohrenstr. 39, \\
	10117 Berlin, Germany
}
\def\emaila{andzhey.koziuk@wias-berlin.de}
\def\emailb{spokoiny@wias-berlin.de}

\def\month{August}
\def\year{2017}

\newcommand{\ifpaper}[2]{#2}	  
\newcommand{\ifau}[3]{#2}			
\newcommand{\style}[2]{#2}			

\documentclass[11pt]{article}

\usepackage{amsmath,amssymb,amsthm}
\usepackage{epsfig,graphicx}
\usepackage{comment}
\usepackage{color}
\usepackage{srcltx}
\usepackage[mathscr]{eucal}
\usepackage[math]{easyeqn}
\usepackage{etoolbox}
\usepackage[nottoc]{tocbibind}
\usepackage{hyperref}
\usepackage{tikz}
\usepackage{pdfsync}
\usepackage{mwe}

\newenvironment{myproof}{\text{}\\$\boldsymbol{Proof:}$\text{}\\}{\text{}\\$\boldsymbol{End\;of\;Proof}$\\\text{}}

\newcommand{\Matrix}[4]{\left(\begin{array}{cc} #1 & #2 \\ #3 & #4 \end{array}\right)}

\newcommand{\Svec}[2]{\left(\begin{array}{c} #1 \\ #2 \end{array}\right)}
\newcommand{\dx}[1]{d#1}

\def\tvv{\widetilde{\vv}} 

\def\svv{\vv^*}  
\def\tvvb{\widetilde{\vv}^{\sbt}} 
\def\svvb{\vv^*_{\sbt}}  
\def\Lnullc{($\boldsymbol{\mathcal{L}_0}$)} 
\def\Lc{($\boldsymbol{\mathcal{L}}$)} 
\def\EDnullc{($\boldsymbol{\mathcal{ED}_0}$)} 
\def\EDtwoc{($\boldsymbol{\mathcal{ED}_2}$)} 
\def\Indc{$\boldsymbol{\mathcal{I}}$} 
\def\EB{($\boldsymbol{\mathcal{EB}}$)\;} 
\def\SmallmBias{$\boldsymbol{\mathcal{SMB}}$} 
\def\IndB{$\boldsymbol{\mathcal{IB}}$} 
\def\EDBnullc{($\boldsymbol{\mathcal{EDB}_0}$)} 
\def\EDBtwoc{($\boldsymbol{\mathcal{EDB}_2}$)} 
\def\Hnull{\mathcal{H}_0} 
\def\Halt{\mathcal{H}_1} 
\def\Hnullb{\mathcal{H}_0^{\sbt}} 
\def\sthetav{\thetav^*}  
\def\setav{\etav^*}  
\def\zetab{\zeta^{\sbt}} 
\def\xivh{\xiv_{\mathcal{H}}} 
\def\diamondsuitb{\diamondsuit^{\sbt}} 
\def\Dpn{\D_0^{\prime}} 
\def\svvp{\vv^{\prime,*}} 
\def\vvp{\vv^{\prime}} 
\def\Dh{\widehat{\D}} 
\def\alphab{\alpha^{\sbt}} 

\def\Tv{\boldsymbol{T}} 
\def\epsv{\boldsymbol{\epsilon}} 
\def\piv{\boldsymbol{\pi}} 
\def\tthetav{\widetilde{\thetav}} 

\def\eqdef{\stackrel{\operatorname{def}}{=}}

\def\Sv{\bf{S}}
\def\xiv{\bf{\xi}}

\def\eps{\varepsilon}

\newcommand{\bb}[1]{\boldsymbol{#1}}

\renewcommand{\hat}[1]{\widehat{#1}}
\renewcommand{\tilde}[1]{\widetilde{#1}}

\def\Gamma{\varGamma}
\def\Pi{\varPi}
\def\Sigma{\varSigma}
\def\Delta{\varDelta}
\def\Lambda{\varLambda}
\def\Psi{\varPsi}
\def\Phi{\varPhi}
\def\Theta{\varTheta}
\def\Omega{\varOmega}
\def\Xi{\varXi}
\def\Upsilon{\varUpsilon}

\def\suml{\sum\limits}

\def\Var{\operatorname{Var}}

\def\argmax{\operatornamewithlimits{argmax}}
\def\argmin{\operatornamewithlimits{argmin}}

\def\D{\boldsymbol{D}}

\def\R{I\!\!R}
\def\E{I\!\!E}
\def\P{I\!\!P}

\def\kappa{\varkappa}

\def\av{\bb{a}}
\def\bv{\bb{b}}

\def\sv{\bb{s}}

\def\uv{\bb{u}}
\def\vv{\bb{v}}

\def\xv{\bb{x}}

\def\Sv{\bb{S}}

\def\Yv{\bb{Y}}
\def\Zv{\bb{Z}}

\def\gammav{\bb{\gamma}}

\def\epsv{\bb{\varepsilon}}
\def\etav{\bb{\eta}}
\def\gammav{\bb{\gamma}}

\def\xiv{\bb{\xi}}

\def\Psiv{\bb{\Psi}}

\definecolor{blue(pigment)}{rgb}{0.2, 0.2, 0.6}
\definecolor{ultramarine}{rgb}{0.07, 0.04, 0.56}
\definecolor{darkspringgreen}{rgb}{0.09, 0.45, 0.27}
\definecolor{hookersgreen}{rgb}{0.0, 0.44, 0.0}
\definecolor{plum(traditional)}{rgb}{0.56, 0.27, 0.52}
\definecolor{purple(html/css)}{rgb}{0.5, 0.0, 0.5}
\definecolor{magenta(dye)}{rgb}{0.79, 0.08, 0.48}

\def\Ind{\operatorname{1}\hspace{-4.3pt}\operatorname{I}}

\def\xivb{\xiv_{\rdb}}

\def\bias{\bb{b}}

\def\thetav{\bb{\theta}}

\def\eps{\epsilon}

\def\vv{\bb{v}}

\def\eps{\varepsilon}

\def\txiv{\tilde{\xiv}}

\def\sbt{\circ}
\def\Pb{\P^{*}}
\def\sbt{\hspace{1pt} \flat}

\def\xivb{\xiv^{\sbt}}

\newcounter{example}[section]
\newcounter{remark}[section]
\numberwithin{equation}{section}
\numberwithin{figure}{section}
\numberwithin{example}{section}
\numberwithin{remark}{section}
\newtheorem{theorem}{Theorem}[section]

\newtheorem{lemma}[theorem]{Lemma}
\newtheorem{corollary}[theorem]{Corollary}

\newtheorem{definition}[theorem]{Definition}
\newtheorem{exmp}[example]{Example}
\newtheorem{rmrk}[remark]{Remark}
\newenvironment{example}{\begin{exmp}\rm}{\end{exmp}}
\newenvironment{remark}{\begin{rmrk}\rm}{\end{rmrk}}

\ifpaper
{
	\title{\thetitle}
	\date{\month, \year \vspace{5cm}}

	\renewenvironment{abstract}
	{\centerline{\textbf{Abstract}}\bigskip
		\begin{center}
			\begin{minipage}{10cm}
				\begin{small}
				}
				{   \end{small}
			\end{minipage}
		\end{center}
		\bigskip
	}

}{
	\title{\thetitle}
	\date{}

	\renewenvironment{abstract}
	{\centerline{\textbf{Abstract}}
		\begin{center}
			\begin{minipage}{15cm}
				\begin{small}
				}
				{   \end{small}
			\end{minipage}
		\end{center}
	}
	
}

\style
{
	\textheight=22cm
	\textwidth=17cm
	\topmargin=0pt
	\oddsidemargin=0.2cm
	\evensidemargin=1cm
	\pagestyle{empty}
	\linespread{1.2}
}{ 
	\textheight=21cm
	\textwidth=17cm
	\topmargin=0pt
	\oddsidemargin=0.2cm
	\evensidemargin=1cm
	\pagestyle{empty}
	\linespread{1.1}
}

\begin{document}
	\thispagestyle{empty}	
	\ifpaper
	{
		\title{\thetitle}
		\ifau{ 
			\author
			{
				\authora
				\ifdef{\thanksa}{\thanks{\thanksa}}{}
				\\[5.pt]
				\addressa \\
				\texttt{ \emaila}
			}
		}
		{  
			\author{
				\authora
				\ifdef{\thanksa}{\thanks{\thanksa}}{}
				\\[5.pt]
				\addressa \\
				\texttt{ \emaila}
				\and
				\authorb
				\ifdef{\thanksb}{\thanks{\thanksb}}{}
				\\[5.pt]
				\addressb \\
				\texttt{ \emailb}
			}
		}
		{   
			\author{
				\authora
				\ifdef{\thanksa}{\thanks{\thanksa}}{}
				\\[5.pt]
				\addressa \\
				\texttt{ \emaila}
				\and
				\authorb
				\ifdef{\thanksb}{\thanks{\thanksb}}{}
				\\[5.pt]
				\addressb \\
				\texttt{ \emailb}
				\and
				\authorc
				\ifdef{\thanksc}{\thanks{\thanksc}}{}
				\\[5.pt]
				\addressc \\
				\texttt{ \emailc}
			}
		}
		\thispagestyle{empty}
		\maketitle
		\pagestyle{myheadings}
		\markboth
		{\hfill \textsc{ \small \theruntitle} \hfill}
		{\hfill
			\textsc{ \small
				\ifau{\runauthora}
				{\runauthora and \runauthorb}
				{\runauthora, \runauthorb, and \runauthorc}
			}
			\hfill}
		\thispagestyle{empty}
		\newpage
		\begin{abstract}
			\theabstract
		\end{abstract}
		\pagenumbering{arabic}

		\par\noindent\emph{AMS 2000 Subject Classification:} Primary \amsPrime. Secondary \amsSec

	}{			
		\ifau{ 
			\author
			{
				\authora
				\ifdef{\thanksa}{\thanks{\thanksa}}{}
				\\[5.pt]
				\addressa \\
				\texttt{ \emaila}
			}
		}
		{  
			\author{
				\authora
				\ifdef{\thanksa}{\thanks{\thanksa}}{}
				\\[5.pt]
				\addressa \\
				\texttt{ \emaila} \\
				\and \\
				\authorb
				\ifdef{\thanksb}{\thanks{\thanksb}}{}
				\\[5.pt]
				\addressb \\
				\texttt{ \emailb}
			}
		}
		{   
			\author{
				\authora
				\ifdef{\thanksa}{\thanks{\thanksa}}{}
				\\[5.pt]
				\addressa \\
				\texttt{ \emaila}
				\and
				\authorb
				\ifdef{\thanksb}{\thanks{\thanksb}}{}
				\\[5.pt]
				\addressb \\
				\texttt{ \emailb}
				\and
				\authorc
				\ifdef{\thanksc}{\thanks{\thanksc}}{}
				\\[5.pt]
				\addressc \\
				\texttt{ \emailc}
			}
		}
		\maketitle
		\begin{abstract}
			\theabstract
		\end{abstract}
		\newpage
		\tableofcontents
		\newpage
		\pagenumbering{arabic}
	}

\section{Introduction}
In the work a non-parametric regression with instrumental variables is considered. A general framework is introduced and identification of a target of inference is discussed. Furthermore, multiplier bootstrap in a general form is considered and justified. Moreover, the procedure is used to test a hypothesis on a target function. 
\section{Identification in non-parametric IV regression}\label{NonParamteric}
\subsection{iid model}
Introduce independent identically distributed observations
\begin{equation}\label{SMPL}
\left(Y_i,X_i,\{W^k_i\}_{k=\overline{1,K}}\right)_{i=\overline{1,n}}\in\Omega
\end{equation}
from a sample set 
\[
\Omega\eqdef\R\otimes\mathbf{Q}\otimes\R^{\otimes K}
\]
on a probability space $(\Omega,\mathcal{F}\left(\Omega\right),\P)$. Let $\mathbf{Q}\subset\R$ be a compact and random variables are respectively coming from $Y_1\in\R$, $X_1\in\mathbf{Q}$ and $W^k_1\in\R$. 

Assume a system of $K+1$ non-linear equations
\begin{equation}\label{MIV}
\begin{cases}
\E W_1^1\left(Y_1-f(X_1)\right)=0,\\
\E W_1^2\left(Y_1-f(X_1)\right)=0,\\
...\\
\E W_1^K\left(Y_1-f(X_1)\right)=0,\\
\int_{\mathbf{Q}}f^2(x)dx=const.
\end{cases}
\end{equation}
\text{}\\
A parametric relaxation of the system introduces a non-parametric bias. For an orthonormal functional basis $\{\psi_j(x):\mathbf{Q}\rightarrow\R\}_{j=\overline{1,\infty}}$ define
\begin{equation}\label{A}
\hat{f}(x)\eqdef\suml_{j=1}^J\psi_j(x)\theta^*_j\eqdef\Psiv(x)^T\sthetav
\end{equation}
such that 
\[
\theta^*_j\eqdef\int_{\mathbf{Q}}f(x)\psi_{j}(x)dx.
\]
Then a substitution $f(x)\rightarrow\hat{f}(x)$ transforms ($\ref{MIV}$) and gives
\begin{equation}\label{AMIV}
\begin{cases}
\E W_1^1\left(Y_1-\hat{f}(X_1)\right)=\delta_1,\\
\E W_1^2\left(Y_1-\hat{f}(X_1)\right)=\delta_2,\\
...\\
\E W_1^K\left(Y_1-\hat{f}(X_1)\right)=\delta_K,\\
\int_{\mathbf{Q}}\hat{f}^2(x)dx=const,
\end{cases}
\end{equation}
with a bias
\begin{equation}\label{NB}
\forall k>0\;\;\;\;\delta_k\eqdef\E W_1^k\left(f(X_1)-\hat{f}(X_1)\right).
\end{equation}
Particular case of ($\ref{AMIV}$) under parametric assumption ($\delta_k=0$) and with a single instrument ($K=1$) is a popular choice of a model with instrumental variables (\cite{DA06},\cite{MJM03}). The system is rewritten as
\begin{equation}\label{SI}
\begin{cases}
\E W_1^1\left(Y_1-\hat{f}(X_1)\right)=0,\\
\int_{\mathbf{Q}}\hat{f}^2(x)dx=const,
\end{cases}
\Rightarrow
\begin{cases}
\etav_1^{*T}\thetav=\E W_1^1Y_1,\\
\suml_{j=1}^J\theta^{2}_j=const
\end{cases}
\end{equation}
with the definition $\etav^{*T}_1\eqdef\left(\E W^1_1\psi_1(X_1),\E W^1_1\psi_2(X_1),...,\E W^1_1\psi_J(X_1)\right)$.
\begin{lemma}\label{LL}
	The statements are equivalent.
	\begin{enumerate}
		\item There exists and unique solution $\sthetav$ to ($\ref{SI}$).
		\item $\exists! \beta>0$ such that $\sthetav=\beta\setav_1$ is a solution of ($\ref{SI}$).
	\end{enumerate}
\end{lemma}
\begin{proof}
	A solution to ($\ref{SI}$) can be represented as
	\[
	\sthetav=\alpha Q_{\perp}\setav_{\perp}+\beta\setav_1
	\]
	for a fixed $\alpha$, $\beta$ and $Q_{\perp}\setav_{\perp}$ such that $\etav_{\perp}^{*T}\setav_1=0$ and $Q_{\perp}$ is a rotation of an orthogonal to $\setav_1$ linear subspace in $\R^J$. If the vector $\sthetav$ is unique then $\alpha$ must be zero otherwise there exist infinitely many distinct solutions ($Q_{\perp}\setav_{\perp}\neq Q^{\prime}_{\perp}\setav_{\perp}$ ). On the other hand for $\alpha=0$ the vector $\sthetav$ is unique. 
\end{proof}
\text{}\\
The second statement helps to obtain exact form of a solution to ($\ref{SI}$)
\begin{equation}\label{STM}
\hat{f}(x)=\beta\suml_{j=1}^J\psi_j(x)\eta^*_{1j}=\frac{\E W^1_1 Y_1}{\suml_{j=1}^{J}\left(\E W_1^1 \psi_j(X_1)\right)^2}\suml_{j=1}^J\psi_j(x)\E W_1^1 \psi_j(X_1).
\end{equation}
Hence, the correlation of instrumental variable $W^1$ with features $X_1$ (note $\eta^*_{1j}=\E W_1^1 \psi_j(X_1)$) identifies $\hat{f}(x)$ (up to a scaling) making the choice of the variable $W^1$ a crucial task. An empirical relaxation to ($\ref{SI}$) in the literature (see \cite{DA06},\cite{MJM03}) closely resembles the following form
\begin{equation}\label{SIVR}
\begin{cases}
\Yv_1=\Zv^T\piv\beta+\epsv_1,\\
\Yv_2=\Zv^T\piv+\epsv_2,
\end{cases}
\end{equation}
for $\Yv_1,\Yv_2,\epsv_1,\epsv_2\in\R^n$, $\Zv\in\R^{J\times n}$, $\piv\in\R^J$, $\beta\in\R$ and
\[
\Svec{\epsv_{1,i}}{\epsv_{2,i}}\sim\mathcal{N}\left(0,\Matrix{\lambda_1}{\rho}{\rho}{\lambda_2}\right)
\]
or alternatively (lemma [$\ref{LL}$])
\[
\begin{cases}
\E W_1^1Y_1=\etav_1^{*T}\thetav^*,\\
\|\etav_1^{*}\|^2=const
\end{cases}
\Rightarrow\;\;\;\;
\begin{cases}
W^1_{1,i}Y_{1,i}=W^1_{1,i}\Psiv^T(X_{1,i})\thetav+\eps_{1,i},\\
\|W_{1,i}\Psiv(X_{1,i})\|^2=W^1_{1,i}\Psiv^T(X_{1,i})\thetav/\beta+\eps_{2,i}
\end{cases}
\]
corresponding to the latter system up to a notational convention
\[
W^1_{1,i}Y_{1,i}\eqdef\Yv_{1,i},\;\;\|W^1_{1,i}\Psiv(X_{1,i})\|^2\eqdef\Yv_{2,i},\;\;W^1_{1,i}\psi_j(X_{1,i})\eqdef \Zv_{ji}\;\;\text{and}\;\;\thetav\eqdef\beta\piv.
\]
The model was theoretically and numerically investigated in a number of papers (see \cite{DA06},\cite{MJM03}) and in the article (see 'Numerical') is used as a numerical benchmark. 

The lemma [$\ref{LL}$] is a special case example of a more general statement on identification in ($\ref{AMIV}$).
\begin{lemma}\label{ID}
	The statements are equivalent.
	\begin{enumerate}
		\item There exists and unique solution $\hat{f}(x)$ to the system ($\ref{AMIV}$).
		\item A solution to ($\ref{AMIV}$) is given by $\hat{f}(x)=\suml_{j=1}^J\psi_j(x)\thetav^{id}_j$ where $\thetav^{id}$ is a solution to an optimization problem
		\begin{equation}\label{opt}
		\thetav^{id}=\argmin_{\xv\in\R^J}\|\xv\|^2\;\text{s.t.}\;\;\begin{cases}
		\etav_1^{*T}\xv=\E W_1^1Y_1-\delta_1,\\
		\etav_2^{*T}\xv=\E W_1^2Y_1-\delta_2,\\
		...,\\
		\etav_K^{*T}\xv=\E W_1^KY_1-\delta_K
		\end{cases}
		\end{equation}
		with $\etav^{*T}_k\eqdef\left(\E W^k_1\psi_1(X_1),\E W^k_1\psi_2(X_1),...,\E W^k_1\psi_J(X_1)\right)$.
	\end{enumerate}
\end{lemma}
\begin{proof}
	The model ($\ref{AMIV}$) turns into
	\begin{equation}\label{P}
	\begin{cases}
	\E W_1^1\left(Y_1-\hat{f}(X_1)\right)=\delta_1,\\
	\E W_1^2\left(Y_1-\hat{f}(X_1)\right)=\delta_2,\\
	...\\
	\E W_1^K\left(Y_1-\hat{f}(X_1)\right)=\delta_K,\\
	\int_{\mathbf{Q}}\hat{f}^2(x)dx=const,
	\end{cases}
	\Rightarrow
	\begin{cases}
	\etav_1^{*T}\thetav=\E W_1^1Y_1-\delta_1,\\
	\etav_2^{*T}\thetav=\E W_1^2Y_1-\delta_2,\\
	...,\\
	\etav_K^{*T}\thetav=\E W_1^KY_1-\delta_K,\\
	\suml_{j=1}^J\theta^{2}_j=const.
	\end{cases}
	\end{equation}
	A solution to ($\ref{P}$) is an intersection of a $J$-sphere and a hyperplane $\R^{J-K}$. If it is unique the hyperplane is a tangent linear subspace to the $J$-sphere and the optimization procedure ($\ref{opt}$) is solved by definition of the intersection point. Conversely, if there exist a solution to the optimization problem then it is guaranteed to be unique as a solution to a convex problem with linear constraints and by definition $\hat{f}(x)$ satisfy ($\ref{AMIV}$).
\end{proof}
\text{}\\
An important identification corollary follows from the lemma [$\ref{ID}$].
\begin{theorem}[Identifiability]\label{IDIV}
	Let $f(x)\in\mathcal{H}\left[\mathbf{Q}\right]$ and instrumental variables $\{W^k\}_{k=\overline{1,K}}$ to be such that
	\[
	\lim_{J\rightarrow\infty}\delta_k=0,
	\]
	then $\exists!\;C_I>0$ s.t. a functions on a ball 
	\[
	\mathcal{F}\eqdef\{\|f\|^2_{\mathbb{L}_2\left[\mathbf{Q}\right]}=C_I\}
	\]
	contain a single solution to ($\ref{MIV}$).
\end{theorem}
\begin{proof}
	In ($\ref{AMIV}$) identifiability is equivalent to $\int_{\mathbf{Q}}f(x)\Psiv(x)dx=\thetav^{id}$ with $\|\thetav^{id}\|<\infty$ (lemma [$\ref{ID}$]) and the approximation converges $\lim_{J\rightarrow\infty}\hat{f}(x)=f(x)$ in complete metric space $\mathcal{H}\left[\mathbf{Q}\right]$ to a solution of
	\[
	\begin{cases}
	\E W_1^1\left(Y_1-\hat{f}(X_1)\right)=\delta_1,\\
	\E W_1^2\left(Y_1-\hat{f}(X_1)\right)=\delta_2,\\
	...\\
	\E W_1^K\left(Y_1-\hat{f}(X_1)\right)=\delta_K,\\
	\int_{\mathbf{Q}}\hat{f}^2(x)dx=const,
	\end{cases}
	\Rightarrow
	\begin{cases}
	\E W_1^1\left(Y_1-f(X_1)\right)=0,\\
	\E W_1^2\left(Y_1-f(X_1)\right)=0,\\
	...\\
	\E W_1^K\left(Y_1-f(X_1)\right)=0,\\
	\int_{\mathbf{Q}}f^2(x)dx=const.
	\end{cases}
	\]
	Then it inherits the equivalence from the lemma [$\ref{LL}$] and the ball
	\[
	\mathcal{F}\eqdef\{\|f\|^2_{\mathbb{L}_2\left[\mathbf{Q}\right]}=C_I\}
	\]
	with $C_I\eqdef\|\thetav^{id}\|^2_{\mathbb{L}_2\left[\mathbf{Q}\right]}<\infty$, contains only a single solution. 
	
	Assume otherwise, there exists $C\neq C_I$ s.t. $\{\|f\|^2_{\mathbb{L}_2\left[\mathbf{Q}\right]}=C\}$ and $\{\|f\|^2_{\mathbb{L}_2\left[\mathbf{Q}\right]}=C_I\}$ contain unique solutions, then they must be distinct as $\{\|f\|^2_{\mathbb{L}_2\left[\mathbf{Q}\right]}=C\}\cap\{\|f\|^2_{\mathbb{L}_2\left[\mathbf{Q}\right]}=C_I\}=\emptyset$. Thus, by definition solutions to a respective parametric relaxations of ($\ref{MIV}$) are unique and distinct for any $J>J_0$ greater than some fixed $J_0$ ($\delta_k^C\neq\delta_k^{C_J}$)
	\[
	\begin{cases}
	\E W_1^1\left(Y_1-\hat{f}(X_1)\right)=\delta^{C}_1,\\
	\E W_1^2\left(Y_1-\hat{f}(X_1)\right)=\delta^{C}_2,\\
	...\\
	\E W_1^K\left(Y_1-\hat{f}(X_1)\right)=\delta^{C}_K,\\
	\int_{\mathbf{Q}}\hat{f}^2(x)dx=C,
	\end{cases}
	\leftrightarrow\;\;\;\;
	\begin{cases}
	\E W_1^1\left(Y_1-\hat{f}(X_1)\right)=\delta^{C_I}_1,\\
	\E W_1^2\left(Y_1-\hat{f}(X_1)\right)=\delta^{C_I}_2,\\
	...\\
	\E W_1^K\left(Y_1-\hat{f}(X_1)\right)=\delta^{C_I}_K,\\
	\int_{\mathbf{Q}}\hat{f}^2(x)dx=C_I.
	\end{cases}
	\]
	Alternatively the lemma [$\ref{ID}$] states that there exist two distinct solutions to the respective optimization problem ($\ref{opt}$). However, in the limit $J\rightarrow\infty$ - $\delta_k^{C_I}\rightarrow 0$ and $\delta_k^C\rightarrow 0$ - optimization objectives coincide contradicting the assumption.
\end{proof}
\begin{rmrk}
	One can trace in the lemma [$\ref{LL}$] as well as in the theorem [$\ref{IDIV}$] that a restriction in $\mathbb{L}_2\left[\mathbf{Q}\right]$ norm in ($\ref{MIV}$) enables identifiability. Otherwise an $\mathbb{L}_q\left[\mathbf{Q}\right]$ norm leads to an ill-posed problem.
\end{rmrk}
\subsection{non-iid model}
Redefine
\begin{equation}\label{NSMPL}
\left(Y_i,X_i,\{W^k_i\}_{k=\overline{1,K}}\right)_{i=\overline{1,n}}\in\Omega=\R\otimes\mathbf{Q}\otimes\R^{\otimes K}
\end{equation}
on a probability space $(\Omega,\mathcal{F}\left(\Omega\right),\P)$. Let $\mathbf{Q}\subset\R$ be a compact, random variables from $Y_i\in\R$, $X_i\in\mathbf{Q}$, $W^k_i\in\R$ and let the observations identify uniquely a solution to the system 
\begin{equation}\label{NMIV}
\forall i=\overline{1,n}\;\;\;\;\begin{cases}
\E W_i^1\left(Y_i-\hat{f}(X_i)\right)=\delta_1,\\
\E W_i^2\left(Y_i-\hat{f}(X_i)\right)=\delta_2,\\
...\\
\E W_i^K\left(Y_i-\hat{f}(X_i)\right)=\delta_K,\\
\int_{\mathbf{Q}}\hat{f}^2(x)dx=C_I.
\end{cases}
\Rightarrow\forall i=\overline{1,n}\;\;\;\;
\begin{cases}
\etav^*_{1,i}\etav_{1,i}^{*T}\thetav=\etav^*_{1,i}Z^i_k\,\\
\etav^*_{2,i}\etav_{2,i}^{*T}\thetav=\etav^*_{2,i}Z^i_k\,\\
...,\\
\etav^*_{K,i}\etav_{K,i}^{*T}\thetav=\etav^*_{K,i}Z^i_k\,\\
\suml_{j=1}^J\theta^{2}_j=C_I.
\end{cases}
\end{equation}
in the particular case with 
\[
\etav^{*T}_{k,i}\eqdef\left( \E W^k_{i}\psi_1(X_{i}),\E W^k_{i}\psi_2(X_{i}),...,\E W^k_{i}\psi_J(X_{i})\right)\;\;\text{and}\;\;Z^i_k\eqdef W_i^kY_i-\delta_k.
\]
Identification in non iid case complicates the fact that $n$ is normally larger than $J$ leading to possibly different identifiability scenarios. Distinguish them based on a rank of a matrix
\begin{equation}\label{R}
r\eqdef rank\left(\suml_{i=1}^n\suml_{k=1}^K\etav^*_{k,i}\etav_{k,i}^{*T}\right)=rank\left(\suml_{i=1}^n\suml^K_{k=1}\E W_i^{k}\Psiv(X_i)\E\Psiv^T(X_i)W_i^{k}\right).
\end{equation}
Note that the rank and, thus, a solution to [$\ref{NMIV}$] depends on a sample size $n$ ($K$ is assumed to be fixed). However, there is no prior knowledge of what $r$ corresponds to the identifiable function $f(x)\in\mathcal{H}\left[\mathbf{Q}\right]$. Therefore, the discussion requires an agreement on the target of inference. 

A way to reconcile uniqueness with the observed dependence is to require the function $f(x)\in\mathcal{H}\left[\mathbf{Q}\right]$ and $r$ to be independent from $n$. The model ($\ref{NMIV}$) makes sense if it points consistently at a single function independently from a number of observations. Define accordingly a target function.
\begin{definition}\label{TF}
	Assume $\exists N\leq\infty$ s.t. $\forall n\geq N$ the rank $r=const$, then call a function $\hat{f}(x)\in\mathcal{H}\left[\mathbf{Q}\right]$ a \textbf{target} if it solves ($\ref{NMIV}$) $\forall n\geq N$.
\end{definition}
\begin{rmrk}
	In the case of $n<N$ a bias between a solution and the target $n>N$ has to be considered.
	However, in the subsequent text it is implicitly assumed that a sample size $n>N$. 
\end{rmrk}
Based on the convention [$\ref{TF}$] introduce a classification:
\begin{enumerate}
	\item Complete model: $\forall J>0$ $\exists N\leq\infty$ s.t. $\forall n>N$ the rank $r=J$.
	\item Incomplete model: $\exists J_1>0$ s.t $\forall J>J_1,n>0$ the rank $r\leq J_1$.
\end{enumerate}
\text{}\\
Identification in the 'incomplete' model is equivalent to the iid case with the notational change for the number of instruments $K\leftrightarrow J_1$ and respective change of $K$ equations with instruments to the $J_1$ equations from ($\ref{NMIV}$). Otherwise 'completeness' of a model allows for a direct inversion of ($\ref{NMIV}$). Generally a complete model is given without the restriction $\mathcal{F}\eqdef\{\|f\|^2_{\mathbb{L}_2\left[\mathbf{Q}\right]}=C_I\}$
\begin{equation}\label{CNMIV}
\forall n>N:\;\forall i=\overline{1,n}\;\;\;\;\begin{cases}
\E W_i^1\left(Y_i-\hat{f}(X_i)\right)=\delta_1,\\
\E W_i^2\left(Y_i-\hat{f}(X_i)\right)=\delta_2,\\
...\\
\E W_i^K\left(Y_i-\hat{f}(X_i)\right)=\delta_K.
\end{cases}
\end{equation}
In this case a natural objective function for an inference is a quasi log-likelihood
\begin{equation}\label{LM}
L(\thetav)\eqdef-\frac{1}{2}\suml_{k=1}^{K}\suml_{i=1}^n\left(Z^i_k-\etav^{iT}_k\thetav\right)^2
\end{equation}
again with
\[
\etav^{iT}_k\eqdef\left( W^k_{i}\psi_1(X_{i}), W^k_{i}\psi_2(X_{i}),..., W^k_{i}\psi_J(X_{i})\right)
\]
and
\[
Z^i_k\eqdef W_i^kY_i-\delta_k.
\]
\section{Testing a linear hypothesis: bootstrap log-likelihood ratio test} 
Introduce an empirical relaxation of the biased ($\ref{AMIV}$) 
\begin{equation}\label{EIV}
\begin{cases}
W_i^1\Psiv^T(X_i)\thetav=W_i^1Y_i-\delta_1+\epsv_{1,i},\\
W_i^2\Psiv^T(X_i)\thetav=W_i^2Y_i-\delta_2+\epsv_{2,i},\\
...\\
W_i^K\Psiv^T(X_i)\thetav=W_i^KY_i-\delta_K+\epsv_{K,i},\\
\|\thetav\|^2=C_I
\end{cases}
\end{equation}
with centered errors $\epsv_{k,i}$. Courtesy of the lemma [$\ref{ID}$], a natural objective function is a penalized quasi log-likelihood
\begin{equation}\label{M1}
L(\thetav)\eqdef\suml_{i=1}^n\ell_i(\thetav)\eqdef-\frac{1}{2}\suml_{k=1}^{K}\suml_{i=1}^n\left(Z^i_k-\etav^{iT}_k\thetav\right)^2-\frac{\lambda\|\thetav\|^2}{2}
\end{equation}
with
\[
\etav^{iT}_k\eqdef\left( W^k_{i}\psi_1(X_{i}), W^k_{i}\psi_2(X_{i}),..., W^k_{i}\psi_J(X_{i})\right)\;\;\text{and}\;\;Z^i_k\eqdef W_i^kY_i-\delta_k.
\]
Maximum likelihood estimator (MLE) and its target are given
\[
\tthetav\eqdef\argmax_{\thetav\in\R^p}L\left(\thetav\right)\;\;\text{and}\;\;\;\sthetav\eqdef\argmax_{\thetav\in\R^p}\E L\left(\thetav\right).
\]
For a fixed projector $\{\Pi\in\R^{J\times J}:\R^J\rightarrow \R^{J_1},\;J_1\leq J\}$ introduce a linear hypothesis and define a log-likelihood ratio test
\[
\Hnull:\;\sthetav\in\{\Pi\thetav=0\},
\]
\[
\Halt:\;\sthetav\in\{\R^p\setminus\{\Pi\thetav=0\}\},
\]
\begin{equation}\label{LRT}
T_{LR}\eqdef\sup_{\thetav} L\left(\thetav\right)-\sup_{\thetav\in \Hnull} L\left(\thetav\right).
\end{equation}
The test weakly converges $T_{LR}\rightarrow\chi^2_{J_1}$ to chi-square distribution (theorem $\ref{SquareRootWilksResult}$) and it is convenient to define a quantile as
\[
z_{\alpha}\;:\;\P\left(\left(T_{LR}-J\right)/\sqrt{J}<z_{\alpha}\right)\geq1-\alpha.
\]
It implies that $\lim_{J\rightarrow\infty}z_{\alpha}=\frac{1}{2}erf^{-1}(1-\alpha)$ and that $z_{\alpha}$ weakly depends on a dimension $\exists C<\infty$ s.t. $\forall J>0$, $z_{\alpha}<C$. 

For a set of re-sampling multipliers 
\[
\{u_i\sim\mathcal{N}\left(1,1\right)\}_{i=\overline{1,n}}
\]
define bootstrap $L^{\sbt}(\thetav)$ conditional on the original data
\[
L^{\sbt}(\thetav)=\suml_{i=1}^n\ell_i(\thetav)u_i\eqdef\suml_{i=1}^n\left(\suml_{k=1}^{K}\left(-\frac{\left(Z^i_k-\etav^{iT}_k\thetav\right)^2}{2}-\frac{\lambda\|\thetav\|^2}{2nK}\right)\right)u_i.
\]
and corresponding bootstrap MLE (bMLE) and its target
\[
\tthetav^{\sbt}\eqdef\argmax_{\thetav\in\R^p}L^{\sbt}\left(\thetav\right)\;\;\text{and}\;\;\;\tthetav\eqdef\argmax_{\thetav\in\R^p}\E L^{\sbt}\left(\thetav\right)=\argmax_{\thetav\in\R^p}L\left(\thetav\right).
\]
A centered hypothesis and a respective test are defined accordingly 
\[
\Hnullb:\;\tthetav\in\{\Pi(\thetav-\tthetav)=0\},
\]
\begin{equation}\label{BLRT}
T_{BLR}\eqdef\sup_{\thetav} L^{\sbt}\left(\thetav\right)-\sup_{\thetav\in \Hnullb} L^{\sbt}\left(\thetav\right).
\end{equation}
And analogously $z^{\sbt}_{\alpha}\;:\;\P^{\sbt}\left(\left(T_{BLR}-J\right)/\sqrt{J}<z^{\sbt}_{\alpha}\right)\geq1-\alpha$. The theorem [$\ref{BootstrapSquareRootWilksResult}$] enables the same convergence in growing dimension $\lim_{J\rightarrow\infty}z^{\sbt}_{\alpha}=\frac{1}{2}erf^{-1}(1-\alpha)$.

Under parametric assumption - $\forall k>0$ the non-parametric bias is zero $\delta_k=0$ - the bootstrap log-likelihood test is empirically attainable and the quantile $z^{\sbt}_{\alpha}$ is computed explicitly. On the other hand an unattainable quantile $z_{\alpha}$ calibrates $T_{LR}$. Between the two exists a direct correspondence. In the section [$\ref{GCA}$] it is demonstrated that $z^{\sbt}_{\alpha}$ can be used instead of $z_{\alpha}$. 
\begin{equation}\label{BP}
\textbf{\large{Multiplier bootstrap procdeure}:}
\end{equation}
\begin{itemize}
	\item Sample $\{u_i\sim\mathcal{N}\left(1,1\right)\}_{i=\overline{1,n}}$ computing $z^{\sbt}_{\alpha}$ satisfying $\P^{\sbt}\left(\left(T_{BLR}-J\right)/\sqrt{J}<z^{\sbt}_{\alpha}\right)\geq1-\alpha$
	\item Test $\Hnull$ against $\Halt$ using the inequalities 
	\[
	\Hnull:\;T_{LR}<J+z^{\sbt}_{\alpha}\sqrt{J}\;\;\text{and}\;\;\Halt:\;T_{LR}>J+z^{\sbt}_{\alpha}\sqrt{J}.
	\]
\end{itemize}
\text{}\\
The idea is numerically validated in the section 'Numerical'. Its theoretical justification follows immediately.

\section{Finite sample theory}\label{FST}
In a most general case neither an objective estimates consistently $f(x)\in\mathcal{H}\left[\mathbf{Q}\right]$ nor a model ($\ref{SMPL}$) is justified as a suitable for arbitrary $L(\thetav)$. Moreover, a regression with instrumental variables adds an additional concern, chosen instruments can be weakly identified (see section [$\ref{CL}$]) and an inference in the problem might involve a separate testing on weakness complicating an original problem. 

Finite sample approach (Spokoiny 2012 \cite{spokoiny2012penalized}) is an option to merry a structure of $L(\thetav)$ with a properties of a probability space ($\ref{SMPL}$) and automatically account for an unknown nature of instruments in a regression problem. 
\begin{equation}\label{FSC}
\textbf{\large{Finite sample theory}:}
\end{equation}
\begin{itemize}
	\item \textbf{[Identifiability]} $\sigma_k^2\eqdef\E\left(Z_k^i-\etav^{iT}_k\sthetav\right)^2<\infty$ and $n\suml_{k=1}^K\left(\sigma_k^2-1\right)\E\etav^1_{k}\etav_{k}^{1T}<\lambda$
	\item \textbf{[Error/IV]} $\forall k$ an error $Z^i_k-\etav^{iT}_k\sthetav$ is independent from $Z^i_k$ and $\etav^{iT}_k$ 
	\item \textbf{[Design]} $\sup_j\|\suml_{k=1}^K\D_0^{-1}\etav^i_{k,j}\|\leq1/2$ with $\D_0^2=\left(n\suml_{k=1}^K\E\etav^1_{k}\etav_{k}^{1T}\right)+\lambda I$
	\item \textbf{[Moments]} $\exists \lambda_0,C_0<\infty$ s.t. $\E e^{\lambda_0\epsilon_i}\leq C_0$ with $\epsilon_i\eqdef\suml_{k=1}^K\left(Z_k^i-\E Z_k^i\right)$
	\item \textbf{[Target]} $\exists N>0$ s.t. for a sample size $\forall n\geq N$ and any subset $A$ of the size $|A|\geq N$ of the index set $\{1,2,3...,n\}$ the solution to  $\suml_{i\in A}\nabla\E\ell_i(\thetav)=0$ is unique.
\end{itemize}
\text{\\}
\begin{rmrk}
	The conditions validate the one from Spokoiny 2012 \cite{spokoiny2012penalized} p. 27 section 3.6 on penalized generalized linear model with the link function $g(v):\R\rightarrow\R$ in the considered case $g(v)\eqdef v^2$. As for the condition 'Target' see the discussion below.
\end{rmrk}
\subsection{Wilks expansion}
The conditions ($\ref{FSC}$) give a ground to statistical analysis of a quasi log-likelihood. An objective function assumes concentration of an estimation $\tthetav$ around the parameter $\sthetav$. Thus, the log-likelihood behavior dominantly depend on a local approximation in the vicinity of the target. Based on the conditions ($\ref{FSC}$) one can derive formally the Wilks expansion (Spokoiny 2012 \cite{spokoiny2012penalized}) for the quasi log-likelihood $L(\thetav)$.
\begin{theorem}\label{W}
	Suppose conditions ($\ref{FSC}$) are fulfilled. Define a score vector
	\[
	\xiv\eqdef\left(\Delta\E L(\sthetav)\right)^{-1/2}\nabla L(\sthetav).
	\]
	then it holds with a universal constant $C>0$
	\[
	\left|\sqrt{2L(\tthetav,\sthetav)}-\|\xiv\|\right|\leq C\left(J+x\right)/\sqrt{Kn}
	\]
	at least with the probability $1-5e^{-x}$.
\end{theorem}
Bootstrap analogue of the Wilks expansion also follows. It was claimed in theorem B.4, section B.2 in Spokoiny, Zhilova 2015 \cite{spokoiny2015bootstrap}.
\begin{theorem}\label{BW}
	Suppose conditions ($\ref{FSC}$) are fulfilled. Define a bootstrap score vector
	\[
	\xivb\eqdef\left(\Delta\E L(\sthetav)\right)^{-1/2}\nabla\left(L^{\sbt}(\sthetav)-L(\sthetav)\right),
	\]
	then it holds with a universal constant $C>0$
	\[
	\left|\sqrt{2L^{\sbt}\left(\tthetav^{\sbt},\tthetav\right)}-\|\xiv^{\sbt}\|\right|\leq C\left(J+x\right)/\sqrt{Kn}
	\]
	at least with the probability $1-5e^{-x}$.
\end{theorem}
Moreover, the log-likelihood statistic follows the same local approximation in the context of hypothesis testing and the $T_{LR}$ satisfies (see appendix - section ($\ref{SquareRootWilksResultsForRealandBootstrapWorldsAndBootstrapValidity}$)). 
\begin{theorem}\label{SquareRootWilksResult}
	Assume conditions ($\ref{FSC}$) are satisfied then with a universal constant $C>0$
	\[
	\left|\sqrt{2T_{LR}}-\|\xiv^s\|\right|\leq C\left(J+x\right)/\sqrt{Kn}
	\]
	with probability $\geq1-Ce^{-x}$. The score vector is defined respectively
	\[
	\xiv^s\eqdef D_0^{-1/2}\left(\nabla_{\Pi\thetav} L(\sthetav)-\left(I-\Pi\right)\Delta\E L(\sthetav)\Pi^T\left(\left(I-\Pi\right)\Delta\E L(\sthetav)\left(I-\Pi\right)^T\right)^{-1}\nabla_{\left(I-\Pi\right)\thetav} L(\sthetav)\right),
	\]
	and Fisher information matrix
	\[
	D_0^2\eqdef -\Pi\Delta\E L(\sthetav)\Pi^T+\left(I-\Pi\right)\Delta\E L(\sthetav)\Pi^T\left(\left(I-\Pi\right)\Delta\E L(\sthetav)\left(I-\Pi\right)^T\right)^{-1}\Pi\Delta\E L(\sthetav)\left(I-\Pi\right)^T.
	\]
\end{theorem} 
Similar statement can be proven in the bootstrap world.
\begin{theorem}\label{BootstrapSquareRootWilksResult}
	Assume conditions ($\ref{FSC}$) are fulfilled  then with probability $\geq1-Ce^{-x}$ holds
	\[
	\left|\sqrt{2T_{BLR}}-\|\xiv^s_{\sbt}\|\right|\leq C\left(J+x\right)/\sqrt{Kn},
	\]
	with a universal constant $C>0$, where a score vector is given
	\[
	\xiv^s_{\sbt}\eqdef D_0^{-1/2}\left(\nabla_{\Pi\thetav} L^{\sbt}(\sthetav)-\left(I-\Pi\right)\Delta\E L(\sthetav)\Pi^T\left(\left(I-\Pi\right)\Delta\E L(\sthetav)\left(I-\Pi\right)^T\right)^{-1}\nabla_{\left(I-\Pi\right)\thetav} L^{\sbt}(\sthetav)\right).
	\]
\end{theorem}
The theorem is effectively the same for $L(\thetav)$ as the re-sampling procedure replicates sufficient for the statement assumptions of a quasi log-likelihood (shown in section $\ref{BootstrapAnaloguesOfRealWorldConditionsAndDeviationBoundsOnScoreVector}$ Appendix).
\subsection{Small Modelling Bias}
In view of the re-sampling justification a separate discussion deserves a small modeling bias from Spokoiny, Zhilova 2015 \cite{spokoiny2015bootstrap}. The condition appears from the general way to prove the re-sampling procedure. Namely, for a small error term it is claimed
\[
\sup_t\left|\P\left(T_{LR}<t\right)-\P\left(T_{BLR}<t\right)\right|\leq \textit{error}+\|H_0^{-1}B_0^2H_0^{-1}\|_{op}
\]
with the matrices
\[
H_0^2=\suml_{i=1}^n\E\nabla\ell_i(\sthetav)\nabla^T\ell_i(\sthetav)\;\;\text{and}\;\;B_0^2=\suml_{i=1}^n\nabla\E\ell_i(\sthetav)\nabla^T\E\ell_i(\sthetav),
\]
where the term $\|H_0^{-1}B_0^2H_0^{-1}\|_{op}$ is assumed to be of the error order essentially meaning that the deterministic bias is small. However, the assumption 
\[
\|H_0^{-1}B_0^2H_0^{-1}\|_{op}\sim\textit{error}
\]
appears in the current development only in the form of the condition 'Target' in ($\ref{FSC}$). The substitution is possible due to the next lemma.
\begin{theorem}
	Assume that the condition 'Target' holds, then $\|H_0^{-1}B_0^2H_0^{-1}\|_{op}=0$.
\end{theorem}
\begin{proof}
	By definition of a target of estimation
	\[
	\suml_{i=1}^N\nabla\E\ell_i(\sthetav_0)=0,\;\;\text{and}\;\;\nabla\E\ell_{j}(\sthetav_1)+\suml_{i=1}^N\nabla\E\ell_i(\sthetav_1)=0.
	\]
	The condition 'Target' implies that $\sthetav=\sthetav_0=\sthetav_1$. Meaning, that any particular choice of the term $\nabla\E\ell_{j}(\sthetav)$ with the index $j\in\{1,2,3...,n\}$ is also zero - $\suml_{i=1}^N\nabla\E\ell_i(\sthetav_0)=\suml_{i=1}^N\nabla\E\ell_i(\sthetav_1)$. Thus, $B_0^2=0$ and the statement follows.
\end{proof}
\section{Gaussian comparison and approximation}
There are two results that constitute a basis for the re-sampling ($\ref{BP}$). The first - Gaussian comparison - is taken from {G{\"o}tze}, F. and {Naumov}, A. and {Spokoiny}, V. and {Ulyanov}, V. \cite{2017arXiv170808663G} and adapted to the needs and notations in the work.
\begin{theorem}\label{GC}
	Assume centered Gaussian vectors $\xiv_0\sim\mathcal{N}(0,\Sigma_0)$ and $\xiv_1\sim\mathcal{N}(0,\Sigma_1)$ then it holds
	\[
	\sup_t\left|\P\left(\|\xiv_1\|<t\right)-\P\left(\|\xiv_0\|<t\right)\right|\leq\sup_{j=\{0,1\}}C\sqrt{Tr\Sigma_j}\|I-\Sigma_0^{-1}\Sigma_1\|_{op}
	\]
	with a universal constant $C<\infty$, where $\|\cdot\|_{op}$ stands for the operator norm of a matrix.
\end{theorem}
The second - Gaussian approximation - has been developed in the appendix (section [$\ref{GAR}$]). 

Introduce the notations for the vectors 
\[
\xiv_1\eqdef\suml_{i=1}^{n}\xiv_{1,i},\;\; \text{and}\;\;\xiv_0\eqdef\suml_{i=1}^{n}\xiv_{0,i}
\]
such that
\begin{enumerate}
	\item $\xiv_{1,i_0}$ and $\xiv_{0,i_1}$ are independent and sub-Gaussian
	\item $\E\xiv_{1}\xiv_{1}^T=\E\xiv_{0}\xiv_{0}^T=\Sigma$.
\end{enumerate}
Then a simplified version of the theorem [$\ref{S1GAR}$] from the appendix holds.
\begin{theorem}\label{SGAR}
	Assume the framework above, then 
	\[
	\sup_{t}\left|\P\left(\|\xiv_1\|<t\right)-\P\left(\|\xiv_0\|<t\right)\right|\leq C\frac{\left(Tr\Sigma\right)^{3/2}}{\sqrt{n}}
	\]
	with the universal constant $C<\infty$.
\end{theorem}
Finally, the critical value $z_{\alpha}$ and the empirical $z^{\sbt}_{\alpha}$ are glued together by a matrix concentration inequalities from the section ($\ref{Matrix}$). 

The essence of the re-sampling is to translate the closeness of $z_{\alpha}$ and $z^{\sbt}_{\alpha}$ into the closeness of the matrices $\E\xiv^s\xiv^{sT}\sim\E\xiv_{\sbt}^s\xiv_{\sbt}^{sT}$  -with the help of the Wilks expansion (theorems [$\ref{SquareRootWilksResult}$,$\ref{BootstrapSquareRootWilksResult}$]) and Gaussian comparison result - and approximate unknown $\xiv_s,\xiv_{\sbt}^s$ by the respective Gaussian counterparts. It all amounts to the central theorem.
\begin{theorem}
	The parametric model ($\ref{AMIV}$) in the introduction - $\delta_k=0$ -  under the assumption ($\ref{FSC}$) enables
	\[
	\left|\P\left(\left(T_{LR}-J\right)/\sqrt{J}>z^{\sbt}_{\alpha}\right)-\alpha\right|\leq C_0\frac{J^{3/2}}{\sqrt{Kn}}+C_1\sqrt{\frac{J\log J+x}{Kn}}
	\]
	with a dominating probability $>1-C_2e^{-x}$ and universal constants $C_0,C_1<\infty$.
\end{theorem}
\begin{rmrk}
	Note that the critical value $z^{\sbt}_{\alpha}$ depends on experimental data at hand and is fixed when the expectation is taken $\E\Ind\left(\left(T_{LR}-J\right)/\sqrt{J}>z^{\sbt}_{\alpha}\right)$ with respect to the data generating $T_{LR}$ statistics.
\end{rmrk}
\section{Numerical: conditional and bootstrap log-likelihood ratio tests}\label{Numerical Study}
Calibrate BLR test on a model from Andrews, Moreira and Stock \cite{DA06}. In the paper the authors proposed conditional likelihood ratio test (CLR - $T_{CLR}$) used here as a benchmark. The simulated model reads as  
\begin{equation}\label{BM}
\Yv_1=\Zv^T\piv\beta+\epsv_1,
\end{equation}
\begin{equation}
\Yv_2=\Zv^T\piv+\epsv_2,
\end{equation}
where $\Yv_1,\Yv_2,\epsv_1,\epsv_2\in\R^n$, $\Zv\in\R^{J\times n}$, $\piv\in\R^J$ and $\beta\in\R$ with a matrix ${\Zv}_{i,j}\eqdef \cos\left(\frac{2\pi i j}{n}\right)$, $\beta^*=1$ and $\piv^*_i\sim i$ (see section 1). And the hypothesis 
\[
H_0:\;\beta^*=\beta_0\;\;\text{against}\;\;H_1:\;\beta^*\neq\beta_0
\]
on a value of a structural parameter $\beta$. For the hypothesis Moreira \cite{MJM03} and later Andrews, Moreira and Stock \cite{DA06} construct a CLR test based on the two vectors
\[
\Sv=(\Zv^T\Zv)^{-\frac{1}{2}}\Zv^T\Yv\bv(\bv^T\Omega\bv)^{-\frac{1}{2}}
\]
and
\[
\Tv=(\Zv^T\Zv)^{-\frac{1}{2}}\Zv^T\Yv\av(\av^T\Omega^{-1}\av)^{-\frac{1}{2}}
\]
with the notations $\Yv\eqdef\left[\Yv_1,\Yv_2\right]$, $\av^T\eqdef\left(\beta_0,1\right)$ and $\bv^T\eqdef\left(1,-\beta_0\right)$. $\Sv$ and $\Tv$ are independent and together present sufficient statistics for the model ($\ref{BM}$) with only $\Tv$ depending on instruments' identification, thus conditioning on $\Tv$ and CLR test. Log-likelihood ratio statistics in ($\ref{BM}$) is represented as (see Moreira 2003 \cite{MJM03}) -
\[
T_{LR}=\Sv^T\Sv-\Tv^T\Tv+\sqrt{(\Sv^T\Sv-\Tv^T\Tv)^2+4(\Sv^T\Tv)^2}.
\]
Additionally Lagrange multiplier and Anderson-Rubin tests are given by
\[
T_{LM}=\frac{(\Sv^T\Tv)^2}{\Tv^T\Tv},
\]
\[
T_{AR}=\frac{\Sv^T\Sv}{J}
\]
The latter two are known to perform acceptably except for weakly identified case.
\\
First, correctly specified model is generated for the sample of $n=200$ and with weak instruments ($\piv^{*T}\Zv\Zv^T\piv^*=\frac{C}{n}$). In this case powers of $T_{BLR}$, $T_{CLR}$ and true $T_{LR}$ tests are drawn on the figure ($\ref{Fig.PowerEnvelopeWeakLRBLRCLR}$). To be consistent $T_{BLR}$ is also compared to $T_{LM}$ and $T_{AR}$. The comparison is given on the figure ($\ref{Fig.PowerEnvelopeWeakLRARLM}$) and the data in the case is aggregated in the table ($\ref{Tab:PowerEnvelopeWeakLRBLRCLRARLM}$).
\\
Moreover an important step is to check how robust $T_{BLR}$ to a misspecification of the model. Three special examples are simulated:
\begin{enumerate}
	\item $\epsv_1,\epsv_2\sim Laplace(0,1)$,
	\item $\epsv_{1i},\epsv_{2i}\sim \mathcal{N}(0,\frac{5i}{n}\Omega)$,
	\item $\epsv_{1i},\epsv_{2i}\sim \mathcal{N}(0,(2+1.5\sin(6\pi i/n))\Omega)$.
\end{enumerate}
\text{}\\
Experiment (1) can be found on the figures ($\ref{WeakMisspecifiedLaplace}$), ($\ref{WeakMisspecifiedLaplaceARLM}$) and in the table ($\ref{Tab:PowerEnvelopeWeakMispecifiedLaplaceLRBLRCLRARLM}$). Numerical study of the experiment (2) with misspecified heteroskedastic error is given on the figure ($\ref{WeakMisspecifiedHeteroskedastic}$)  and collected in the table ($\ref{Tab:PowerEnvelopeWeakMispecifiedHeteroskedasticLRBLRCLRARLM}$). The last experiment is shown on the figure ($\ref{WeakMisspecifiedHeteroskedasticPeriodic}$) and in the table ($\ref{Tab:PowerEnvelopeWeakMispecifiedHeteroskedasticPeriodicLRBLRCLRARLM}$). 
\begin{rmrk}
	All the figures and tables are collected in the end of the work.
\end{rmrk}
\section{Strength of instrumental variables}
On practice one wants to distinguish instruments based on its strength. For the clarity of exposition the section considers a simplified log-likelihood ($\ref{LM}$) identifying complete model with the Fisher information matrix
\[
\D_0^2=-\Delta\E L(\sthetav)=\suml_{i=1}^n\suml^K_{k=1}\E\etav^*_{ki}\etav^{*T}_{ki}=\suml_{i=1}^n\suml^K_{k=1}\E W_i^{k}\Psiv(X_i)\Psiv^T(X_i)W_i^{k}.
\]
Weak instrumental variables introduce an unavoidable lower bound on estimation error (lemma [$\ref{WIV}$], see the proof in the appendix ($\ref{SWIV}$)). 
\begin{lemma}\label{WIV}
	Let conditions ($\ref{FSC}$) hold then
	\[
	\exists N>0,\;\text{s.t.}\;\;\forall n>N\;\;\E\|\widetilde{\thetav}-\sthetav\|^2\geq \frac{C_J}{\sup_{\|\uv\|=1}\suml_{i=1}^n\suml_{k=1}^K\E\left(\uv^T\Psiv^(X_i)W_i^{k}\right)^2},
	\]
	with a factor $C_J>0$ depending on dimensionality $J$.
\end{lemma}
\text{}\\
In view of a hypothesis testing it amounts to an indifference region of a test (see the section 'Numerical').\\
\textbf{Classification of Instrumental Variables}:\label{CL}
\begin{enumerate}
	\item Weak instruments
	\[
	\sup_{\|\uv\|=1}\suml_{i=1}^n\suml_{k=1}^K\E\left(\uv^T\Psiv(X_i)W_i^{k}\right)^2\sim  K/C\;\;\text{and}\;\;\E\|\tthetav-\sthetav\|^2\geq \frac{CC_J}{K}
	\]
	\item Semi-strong instruments with $0<\alpha<1$
	\[
	\sup_{\|\uv\|=1}\suml_{i=1}^n\suml_{k=1}^K\E\left(\uv^T\Psiv(X_i)W_i^{k}\right)^2\sim  K n^{\alpha}/C\;\;\text{and}\;\;\E\|\tthetav-\sthetav\|^2\geq\frac{CC_J}{Kn^{1-\alpha}}
	\]
	\item Strong instruments 
	\[
	\sup_{\|\uv\|=1}\suml_{i=1}^n\suml_{k=1}^K\E\left(\uv^T\Psiv(X_i)W_i^{k}\right)^2\sim  Kn/C\;\;\text{and}\;\;\E\|\tthetav-\sthetav\|^2\geq\frac{CC_J}{Kn}
	\]
\end{enumerate}
Weak instruments effectively cancel an analysis based on a limiting distribution of a test statistics. Therefore, an IV regression requires a treatment under the finite sample assumption. 
\section{Appendix}\label{ModelsAppendix}
\subsection{Classification of instrumental variables}\label{SWIV}
\begin{lemma}
	Let conditions ($\ref{FSC}$) hold then
	\[
	\exists N>0,\;\text{s.t.}\;\;\forall n>N\;\;\E\|\widetilde{\thetav}-\sthetav\|^2\geq \frac{C_J}{\sup_{\|\uv\|=1}\suml_{i=1}^n\suml_{k=1}^K\E\left(\uv^T\Psiv^(X_i)W_i^{k}\right)^2},
	\]
	with a factor $C_J>0$ depending on dimensionality $J$.
\end{lemma}
\begin{proof}
	Fisher expansion (Spokoiny \cite{spokoiny2012penalized}) on the set of dominating probability $\P\left(\Upsilon\right)>1-C e^{-x}$ is written as
	\[
	\|\D_0\left(\tthetav-\sthetav\right)-\xiv\|\leq C(J+x)/\sqrt{n}.
	\]
	with the matrix $\D_0^2=\suml_{i=1}^n\suml^K_{k=1}\E W_i^{k}\Psiv(X_i)\Psiv^T(X_i)W_i^{k}$. Introduce also an inequality 
	\[
	\|\xiv\|^2\leq\left(\|\xiv-\D_0\left(\tthetav-\sthetav\right)\|+\|\D_0\left(\tthetav-\sthetav\right)\|\right)^2\leq2\|\xiv-\D_0\left(\tthetav-\sthetav\right)\|^2+2\|\D_0\left(\tthetav-\sthetav\right)\|^2.
	\]
	It gives 
	\[
	\E\|\D_0(\tthetav-\sthetav)\|^2\geq\E_{\Upsilon}\|\D_0(\tthetav-\sthetav)\|^2\geq\frac{1}{2}\E_{\Upsilon}\|\xiv\|^2-C(J+x)^2/n
	\]
	and the inquired statement follows with $N>0$ s.t. $\inf_N\{\frac{1}{2}\E_{\Upsilon}\|\xiv\|^2-C(J+x)^2/N>0\}$ and a constant $C_J\eqdef\frac{1}{2}\E_{\Upsilon}\|\xiv\|^2-C(J+x)^2/N$. 
\end{proof}
\subsection{Non-parametric bias}\label{AppendixBias}
The bias term - $\bv_{J}\eqdef\|\hat{\thetav}-\sthetav\|$ - between parametric and non-parametric functions from the model in chapter $\ref{NonParamteric}$ is quantified in the lemma.
\begin{lemma}\label{DetermenisticNon-parametricError}
	Assume that basis functions $\psi_j$ follow -
	\[
	\psi_j^{(s)}\leq j^{s}\psi_j
	\]
	with some positive constant $s3$. Let $f(x)$ be s.t. $f\in\mathcal{S}^s$ where 
	\[
	\mathcal{S}^s\eqdef\{f:||D^{s}g||\leq C_{f}\},
	\]
	with the notation $D^{s}(\cdot)\eqdef\frac{\partial^{s}}{\partial x}(\cdot)$. Then the bias satisfies
	\[
	\bv_{J}=\|\hat{\thetav}-\sthetav\|\leq C_{f}J^{-s}.
	\]
\end{lemma}
\begin{myproof}
	Straightforwardly using smoothness of functions from a Sobolev class it can be argued for $s<\infty$ that -
	\[
	J^{s}\|\hat{\thetav}^*-\sthetav\|=J^{s}\|\suml_{j}\theta^*_j\psi_j-\suml_{j\leq J}\theta^*_j\psi_j\|\leq\|\suml_{j}\theta^*_{j}j^{s}\psi_j\|\leq\|D^sf\|\leq C_{f}
	\]
	and the result follows.
\end{myproof}
\subsection{Re-sampled quasi log-likelihood}\label{BootstrapAnaloguesOfRealWorldConditionsAndDeviationBoundsOnScoreVector}
A basis for the statistical investigation of a re-sampled log-likelihood builds on the probabilistic equivalence with an original quasi log-likelihood. In the section one also uses notations from Spokoiny 2012 \cite{spokoiny2012penalized}.

An analogue to \EDnullc\; condition for re-sampled log-likelihood - will be referred to as \EDBnullc\; - readily follows from normality of re-sampling weights $\{u_i\}_{i=\overline{1,n}}$.
\begin{lemma}\label{EDB_0}
	Suppose that conditions ($\ref{FSC}$) are justified, then there exist a positive symmetric matrix $V_0$ and constants $\nu_0\geq1$ and $g\geq0$ such that $\Var\left(\nabla\zeta(\sthetav)\right)\leq V_0^2$ and
	\[
	\forall\|\gammav\|=1\;\;\log\E^{\sbt}\exp\left(\lambda\frac{\gammav^T\nabla\zetab(\sthetav)}{\|V_0\gammav\|}\right)\leq\frac{\nu_0^{2}\lambda^2}{2},\;|\lambda|\leq g
	\]
	with probability $\geq1-e^{-x}$.
\end{lemma} 
\begin{proof}
	Define a vector
	\[
	\sv_i\eqdef\frac{\nabla\ell_i(\sthetav)}{\|V_0\gammav\|},
	\]	
	then using normality of re-sampling weights $u_i$ rewrite  
	\[
	\log\E^{\sbt}\exp\left(\lambda\frac{\gammav^T\nabla\zetab(\sthetav)}{\|V_0\gammav\|}\right)=
	\log\E^{\sbt}\exp\left(\frac{\lambda}{\|V_0\gammav\|}\left(\gammav^T\suml_{i=1}^{n}\nabla\ell(Y_i,\sthetav)(u_i-1)\right)\right)=
	\]
	\[
	=\log\E^{\sbt}\exp\left(\suml_{i=1}^{n}\lambda\gamma^T\sv_i(u_i-1)\right)\leq\frac{\nu_0^2\lambda^2}{2}\suml_{i=1}^{n}\left(\gamma^T\sv_i\right)^2\leq\frac{\nu_0^{\prime2}\lambda^2}{2},
	\]
	where $\nu_0^{\prime}\eqdef\sqrt{\nu_0^2+C\delta}$ for some positive constant $C>0$ and small $\delta$. The last inequality is derived using $\suml_{i=1}^{n}\left(\gamma^T\sv_i\right)^2\leq1+C\delta$ from definition of $V_0$ and matrix concentration inequality (thm [$\ref{MatrixBernstein2}$]).
\end{proof}
Re-sampling analogue to the condition \EDtwoc\;  (Spokoiny 2012 \cite{spokoiny2012penalized}) also follows.
\begin{lemma}\label{EDB_2}
	Let conditions ($\ref{FSC}$) hold true then there exist a positive value $\omega_1(r)\eqdef\sqrt{4\nu_0^2\omega^2x+\frac{2C_{\delta}^2(r)}{n}}$ and for each $r\geq0$, a 
	constant $g(r)\geq0$ such that it holds for any $\vv\in\Upsilon(r)$
	\[
	\log\E^{\sbt}\exp\left(\frac{\lambda}{\omega_1(r)}\frac{\gammav_1^T\nabla^2\zetab(\sthetav)\gammav_2}{\|\D_0\gammav_1\|\;\|\D_0\gammav_2\|}\right)\leq\frac{\nu_0^2\lambda^2}{2},\;|\lambda|\leq g(r).
	\]
\end{lemma}
\begin{proof}
	Here it is convenient to reformulate conditions \Lc\; and \EDtwoc.  Bound on deterministic covariance structure can be rewritten as
	\[
	\|\D_0^{-1}(\D^2(\thetav)-\D_0^{2})\D_0^{-1}\|=\|\D_0^{-1}(-\suml_{i=1}^{n}\nabla^2\E\ell(Y_i,\thetav)-\D_0^{2})\D_0^{-1}\|=
	\]
	\[
	=\|\suml_{i=1}^{n}(\D_0^{-1}\nabla^2\E\ell(Y_i,\thetav)\D_0^{-1}+\frac{I_p}{n})\|=n\|\D_0^{-1}\nabla^2\E\ell(Y_i,\thetav)\D_0^{-1}+\frac{I_p}{n}\|\leq\delta(r),
	\]
	and it follows
	\[
	\|\D_0^{-1}\nabla^2\E\ell(Y_i,\thetav)\D_0^{-1}\|\leq\frac{C_{\delta}(r)}{n}.
	\]
	Next, rewrite \EDtwoc\; mostly in the same fashion, so that it is capable to quantify $\D_0^{-1}\nabla^2\zeta_i(\thetav)\D_0^{-1}$. It follows
	\[
	\log\E\exp\{\frac{\lambda}{\omega}\frac{\gamma_1^T\nabla^2\zeta(\thetav)\gamma_2}{\|\D_0\gamma_1\|\;\|\D_0\gamma_2\|}\}=\log\E\exp\{\frac{\lambda}{\omega}\suml_{i=1}^{n}\frac{\gamma_1^T\nabla^2\zeta_i(\thetav)\gamma_2}{\|\D_0\gamma_1\|\;\|\D_0\gamma_2\|}\}=
	\]
	\[
	=n\log\E\exp\{\frac{\lambda}{\omega}\frac{\gamma_1^T\nabla^2\zeta_i(\thetav)\gamma_2}{\|\D_0\gamma_1\|\;\|\D_0\gamma_2\|}\},
	\]
	where $\zeta_i(\vv)=\ell(Y_i,\thetav)-\E\ell(Y_i,\thetav)$. This means that component-wise \EDtwoc\; condition holds true, namely that
	\[
	\sup_{\gamma_1,\gamma_2\in\R^p}\log\E\exp\{\frac{\lambda}{\omega}\frac{\gamma_1^T\nabla^2\zeta_i(\thetav)\gamma_2}{\|\D_0\gamma_1\|\;\|\D_0\gamma_2\|}\}\leq\frac{\nu_0^2\lambda^2}{2n},\;|\lambda|\leq g(r).
	\]
	The two constitute the substance of the proof. Define complementary variables $s_i\eqdef\frac{\gamma_1^T\nabla^2\ell(Y_i,\thetav)\gamma_2}{\|\D_0\gamma_1\|\;\|\D_0\gamma_2\|}$ and rewrite
	\[
	\log\E^{\sbt}\exp\{\frac{\lambda}{\omega_1(r)}\frac{\gamma_1^T\nabla^2\zetab(\thetav)\gamma_2}{\|\D_0\gamma_1\|\;\|\D_0\gamma_2\|}\}=\log\E^{\sbt}\exp\{\suml_{i=1}^{n}\frac{\lambda}{\omega_1(r)}s_i(u_i-1)\}\leq\frac{\nu_0^2\lambda^2}{2\omega_1^2(r)}\suml_{i=1}^{n}s_i^2.
	\]
	To claim the statement it is sufficient to limit sum $\suml_{i=1}^{n}s_i^2$. Once again rewrite this sum using mentioned above \Lc\; -
	\[
	\suml_{i=1}^{n}s_i^2=\suml_{i=1}^{n}\left(\frac{\gamma_1^T\nabla^2\zeta_i(\thetav)\gamma_2}{\|\D_0\gamma_1\|\;\|\D_0\gamma_2\|}+\frac{\gamma_1^T\nabla^2\E\ell(Y_i,\thetav)\gamma_2}{\|\D_0\gamma_1\|\;\|\D_0\gamma_2\|}\right)^2\leq
	\]
	\[
	\leq2\suml_{i=1}^{n}\left(\frac{\gamma_1^T\nabla^2\zeta_i(\vv)\gamma_2}{\|\D_0\gamma_1\|\;\|\D_0\gamma_2\|}\right)^2+2\suml_{i=1}^{n}\left(\frac{\gamma_1^T\nabla^2\E\ell(Y_i,\thetav)\gamma_2}{\|\D_0\gamma_1\|\;\|\D_0\gamma_2\|}\right)^2\leq
	\]
	\[
	\leq2\suml_{i=1}^{n}\left(\frac{\gamma_1^T\nabla^2\zeta_i(\thetav)\gamma_2}{\|\D_0\gamma_1\|\;\|\D_0\gamma_2\|}\right)^2+\frac{2C_{\delta}^2(r)}{n}
	\]
	The left term in the sum is bounded under \EDtwoc\; and Markov exponential inequality
	\[
	\P\left(\frac{\gamma_1^T\nabla^2\zeta_i(\thetav)\gamma_2}{\|\D_0\gamma_1\|\;\|\D_0\gamma_2\|}\leq t
	\right)\leq\E\exp\{\frac{\lambda^{\prime}\gamma_1^T\nabla^2\zeta_i(\thetav)\gamma_2}{\omega\|\D_0\gamma_1\|\;\|\D_0\gamma_2\|}-\frac{\lambda^{\prime}t}{\omega}\}\leq
	\]
	\[
	\leq\exp\{\frac{\nu_0^2\lambda^{\prime2}}{2n}-\frac{\lambda^{\prime}t}{\omega}\}\leq\exp\{-\frac{t^2n}{2\nu_0^2\omega^2}\},
	\]
	and
	\[
	\P\left(\frac{\gamma_1^T\nabla^2\zeta_i(\thetav)\gamma_2}{\|\D_0\gamma_1\|\;\|\D_0\gamma_2\|}\leq\nu_0\omega\sqrt{\frac{2x}{n}}\right)\leq e^{-x}.
	\]
	Therefore it holds 
	\[
	\suml_{i=1}^{n}s_i^2\leq4\nu_0^2\omega^2x+\frac{2C_{\delta}^2(r)}{n},
	\]
	and now we can see that controlling $\omega_1(r)$ in the way -
	\[\omega_1(r)\eqdef\sqrt{4\nu_0^2\omega^2x+\frac{2C_{\delta}^2(r)}{n}},\]
	justifies inquired in the theorem inequality.
\end{proof}
The lemma in turn helps to bound a stochastic part of re-sampled log-likelihood. The demonstrated equivalence allows to translate statements for log-likelihood  into the re-sampled counterpart.

A result requiring only \EDnullc\; is the deviation bound on $\|\xiv\|$ (Spokoiny Zhilova 2013 \cite{spokoiny2013sharp}). In the work of Spokoiny \cite{spokoiny2012penalized} it has been proven.
\begin{theorem}\label{DeviationOnScore}
	Let \EDnullc\; is fulfilled, then for $g\geq\sqrt{2tr(\D_0^{-1}V_0^2\D_0^{-1})}$, where $V_0^2\geq\Var\nabla\zeta(\sthetav)$ it holds:
	\[
	\P(\|\xiv\|^2\geq\mathfrak{z}^2(x,\D_0^{-1}V_0^2\D_0^{-1}))\leq2e^{-x}+8.4e^{-x_c},
	\] 
	for function $\mathfrak{z}^2(x,\D_0^{-1}V_0^2\D_0^{-1})$ and small positive constant $x_c$ (thm $\ref{DeviationOnBootstrapScore}$).
\end{theorem}
Let us claim the same for $\|\xivb\|$ using  the lemma [$\ref{EDB_0}$].
\begin{theorem}\label{DeviationOnBootstrapScore}
	Let \EDBnullc\; is fulfilled, then for $g\geq\sqrt{1+\frac{C\delta}{g}}\sqrt{2tr(\D_0^{-1}V_0^2\D_0^{-1})}$, where $V_0^2\geq\Var\{\nabla\zeta(\sthetav)\}$ it holds with dominating probability:
	\[
	\P^{\sbt}(\|\xivb\|^2\geq\mathfrak{z}^2(x,\D_0^{-1}V_0^2\D_0^{-1}))\leq2e^{-x}+8.4e^{-x_{c_1}},
	\] 
	for function $\mathfrak{z}^2(x,\D_0^{-1}V_0^2\D_0^{-1})$ and small positive constant $x_{c_1}$, specified below.
\end{theorem}
The function $\mathfrak{z}(x,X)$, where $x\in\R$ and $X\in\R^{p\times p}$, is given by the following formula
\[
\mathfrak{z}^2(x,X)\eqdef\begin{cases} tr(X^2)+\sqrt{8tr(X^4)x},\;x\leq\frac{\sqrt{2tr(X^4)}}{18\lambda_{max}(X^2)}\\ tr(X^2)+6x\lambda_{max}(X^2),\;\frac{\sqrt{2tr(X^4)}}{18\lambda_{max}(X^2)}\leq x\leq x_c\\|z_c+2(x-x_c)/g_{c}|^2\lambda_{max}(X^2),\;x\geq x_c,\end{cases}
\]
where in tern numerical constants $x_c,z_c,g_c$ are defined as follows
\[
2x_c\eqdef2z_c^2/3+\log\det(I_p-2X^2/3\lambda_{max}(X^2))
\]
\[
z_c^2\eqdef(9g^2/4-3tr(X^2)/2)/\lambda_{max}(X^2)
\]
\[
g_c\eqdef\sqrt{g^2-2tr(X^2)/3}/\sqrt{\lambda_{max}(X^2)}.
\]
This technical result is used extensively for the proof of squared-root Wilks result.

Another key result is that\;\EDtwoc\;condition justifies a bound on stochastic part of log-likelihood function. The fact formally is stated in the next theorem.
\begin{theorem}\label{BoundOnStochasticPart}
	Let \EDtwoc\; and (\Indc) hold then $\forall\vv\in\R^p$ following inequality is fulfilled
	\[\|\D_0^{-1}\nabla\zeta(\thetav,\sthetav)\|\leq6\nu_0\omega\mathfrak{Z}(x) r.\]
	Also $\forall\thetav_1,\thetav_2\in\R^p$ holds 
	\[
	\|\D_0^{-1}\nabla\zeta(\thetav_1,\thetav_2)\|\leq12\nu_0\omega\mathfrak{Z}(x) r,
	\] 
	where $\mathfrak{Z}(x)$ is defined as
	\[
	\mathfrak{Z}(x)\eqdef\begin{cases}\mathbb{H}_1+\sqrt{2x}+g^{-1}(g^{-2}x+1)\mathbb{H}_2,&\\ \sqrt{\mathbb{H}_2+2x},& if\;\mathbb{H}_2+2x\leq g^2,\\ g^{-1}x+\frac{1}{2}(g^{-1}\mathbb{H}_2+g), & if\;\mathbb{H}_2+2x\ge g^2. \end{cases}	
	\]
	Here $\mathbb{H}_2=4p$ and $\mathbb{H}_1=2p^{\frac{1}{2}}$; see theorem A.15 in \cite{spokoiny2012penalized}.
\end{theorem}
Let us provide a proof of that statement.
\begin{proof}
	Consider quantity $\|\D_0^{-1}\nabla\zeta(\thetav,\sthetav)\|$ and rewrite it as $\|\D_0^{-1}\nabla^2\zeta(\thetav^{\prime})\D_0\D_0^{-1}(\thetav-\sthetav)\|$, then introducing vector $Y(s)\eqdef\D_0^{-1}\nabla\zeta(\thetav,\sthetav)$, where $s\eqdef\D_0(\thetav-\sthetav)$ we can see that $\nabla_{s}Y(s)=\D_0^{-1}\nabla^2\zeta(\thetav^{\prime})\D_0^{-1}$ and from \EDtwoc\;, which holds for $\nabla_{s}Y(s)$, we have for stochastic process $Y(s)$ by an argument from Spokoiny \cite{spokoiny2012penalized} 
	\[
	\sup_{\thetav\in\Upsilon(r_0)}\|\D_0^{-1}\nabla\zeta(\thetav,\sthetav)\|\leq6\nu_0\mathfrak{Z}(x)\omega r,
	\]
	which is generally drawn from empirical processes theory. Furthermore, one can use triangle inequality to generalize result 
	\[\|\D_0^{-1}\nabla\zeta(\thetav_1,\thetav_2)\|\leq\|\D_0^{-1}\nabla\zeta(\thetav_1,\sthetav)\|+\|\D_0^{-1}\nabla\zeta(\thetav_2,\sthetav)\|\leq12\nu_0\mathfrak{Z}(x)\omega r,\]
	and finalize the proof. 
\end{proof}
Once again it is obviously translated using \EDBtwoc,  justified by the lemma $\ref{EDB_2}$. Therefore, formally one comes at the theorem.
\begin{theorem}\label{BootstrapBoundOnStochasticPart}
	Let \EDBtwoc\;hold true then $\forall\thetav\in\R^p$ following inequality is fulfilled
	\[\|\D_0^{-1}\nabla\zetab(\thetav,\sthetav)\|\leq6\nu_0\omega_1(r)\mathfrak{Z}(x) r,\]
	and $\forall\thetav_1,\thetav_2\in\R^p$ also holds
	\[
	\|\D_0^{-1}\nabla\zetab(\thetav_1,\thetav_2)\|\leq12\nu_0\omega_1(r)\mathfrak{Z}(x) r.
	\] 
\end{theorem}
\subsection{Concentration of MLE and bMLE}\label{ConcentratioPropertiesOfMLEAndbMLEAndBootstrapValidity}
This is technical part of the paper and thus the full version of theorems without unnecessary simplifications is presented. An important result in the section $\ref{ConcentrationPropertiesBootstrapValidity}$ is formulated by the theorem below.
\begin{theorem}
	Let conditions \Lnullc, \Lc, \EDnullc, \EDtwoc, (\Indc), (\EB), (\SmallmBias) and (\IndB) hold true, then for $r_0$ such that following inequalities are fulfilled simultaneously
	\[
	\begin{cases}
	b(r)r\geq2\mathfrak{z}(x,B)\vee4\sqrt{tr(\suml_{i=1}^n\E\xiv_i\xiv_i^T)}+24\nu_0\omega\mathfrak{Z}(x+\log{\frac{2r}{r_0}}),\\
	b(r)r\geq3\mathfrak{z}(x,B)+12\nu_0\mathfrak{Z}(x+\log\frac{2r}{r_0})(\omega+\omega_1(r)),
	\end{cases}
	\]
	where $B\eqdef\D_0^{-1}\Var\{\nabla L(\svv)\}\D_0^{-1}$ following inequalities are fulfilled
	\begin{enumerate}
		\item $\P(\tthetav\notin\Upsilon(r_0))\leq C_1e^{-x},$
		\item $\Pb(\tthetav^{\sbt}\notin\Upsilon(r_0))\leq C_2e^{-x}.$
	\end{enumerate}
	Up to constants and quantities smaller than $\sqrt{\frac{p}{n}}$ the concentration radii follows $r_0\sim\sqrt{p+x}$
\end{theorem}
We will utilize uniform version of local deviation bound on stochastic processes $\nabla\zeta(\thetav)$ and $\nabla\zetab(\thetav)$ from theorems $\ref{BoundOnStochasticPart}$ and $\ref{BootstrapBoundOnStochasticPart}$  and also bounds on $\|\xiv\|$ outlined in previous section to prove this result.
\begin{proof}
	Let us list the key facts needed in the proof in an informal fashion to get an idea of the background required.	
	\begin{enumerate}
		\item \Lc\; condition to bound deterministic part of log-likelihood function\\
		$-2\E L(\vv,\svv)\geq b(r)r^2$
		\item Uniform bound on stochastic processes $\nabla\zeta(\vv)$ and $\nabla\zetab(\vv)$ \\
		$|\zeta(\vv,\svv)-(\vv-\svv)\nabla\zeta(\svv)|\leq\rho(x,r)r$\\
		$|\zetab(\vv,\svv)-(\vv-\svv)\nabla\zetab(\svv)|\leq\rho_1(x,r)r$
		\item Deviation bound on $\|\xiv\|$ and $\|\xivb\|$\\
		$\|\xiv\|\geq\mathfrak{z}(x,B)$\\
		$\|\xivb\|\geq\mathfrak{z}(x,B)$
	\end{enumerate}
	These are sufficient to prove results number one and two in the theorem. Let us divide the proof in parts accordingly to the results provided in the statement.
	
	$\boldsymbol{1.Real\; world\; concentration\; of\; MLE}$
	
	Notice that an inequality $L(\tthetav,\sthetav)\geq0$ always hold and thus by definition binds MLE $\tthetav$ structurally to $\sthetav$. So, if one justifies that there exist minimum $r_0$ such that for $r\geq r_0$ the property breaks than one can claim that $\tthetav$ concentrates within $\Upsilon(r_0)$. Therefore, one need to have a uniform bound on log-likelihood function. Spokoiny \cite{spokoiny2012penalized} has proven that with dominating probability - 
	\[
	|\zeta(\thetav,\sthetav)-(\thetav-\sthetav)\zeta(\sthetav)|\leq\rho(x,r)r,
	\]
	where $\rho(x,r)\eqdef6\nu_0\mathfrak{Z}(x+\log\frac{2r}{r_0})\omega$. Local analogue of which is to be proved in the next section. Then using theorem ($\ref{DeviationOnScore}$) and condition \Lc\; it is possible to see that $r_0$ satisfies
	\[
	b(r)r\geq2\mathfrak{z}(x,B)+2\rho(x,r),
	\]
	then $L(\thetav,\sthetav)$ is most probably ($\geq1-3e^{-x}$) less then zero.
	
	$\boldsymbol{2. Bootstrap\; world\; concentration\; of\; bMLE}$
	
	Interestingly in the bootstrap world one needs to extend the set where $\tthetav^{\sbt}$ concentrates. However, the key idea of the proof remains. 
	
	By definition $\L^{\sbt}(\tvvb,\svvb)$ is positive. A uniform bound on $\zetab(\thetav,\sthetav)$ over $\R^p\setminus\Upsilon(r_0)$ translates as
	\[
	|\zetab(\thetav,\tthetav)-(\thetav-\tthetav)\nabla\zetab(\tthetav)|\leq\rho_1(x,r)r,
	\]
	where $\rho_1(x,r)\eqdef6\nu_0\mathfrak{Z}(x+\log\frac{2r}{r_0})\omega_1(r)$. Rewriting it one has
	\[
	|\L^{\sbt}(\thetav,\tthetav)-L(\thetav,\sthetav)-L(\tthetav,\sthetav)-(\thetav-\tthetav)\nabla\zetab(\tthetav)|\leq\rho_1(x,r)r,
	\]
	and the deviation bound on $\|\xivb\|$, from theorem $\ref{DeviationOnBootstrapScore}$, and part one of the proof enable with probability $\geq1-3e^{-x}$ an inequality
	\[
	|L(\thetav,\sthetav)|\leq\rho(r,x)r+r\mathfrak{z}(x,B)-\frac{r^2b(r)}{2}.
	\]
	And $\L^{\sbt}(\tthetav^{\sbt},\tthetav)$ is negative for $r^{\sbt}_0$ satisfying inequality
	\[
	b(r)r\geq12\nu_0\mathfrak{Z}(x+\log\frac{2r}{r_0})(\omega+\omega_1(r))+3\mathfrak{z}(x,B).
	\] 
\end{proof}

\subsection{Square root Wilks expansion}\label{SquareRootWilksResultsForRealandBootstrapWorldsAndBootstrapValidity}
\begin{theorem}\label{SQW}
	Let conditions ($\ref{FSC}$) to be fulfilled,  then with probability $\geq1-Ce^{-x}$ holds
	\[
	\left|\sqrt{2T_{LR}}-\|\xiv^s\|\right|\leq7\diamondsuit(r_0,x),
	\]
	where $\diamondsuit(r,x)$ is given by
	\[
	\diamondsuit(r,x)\eqdef(\delta(r)+6\nu_0\omega\mathfrak{Z}(x))r.
	\]
\end{theorem}
\begin{proof}
	Compared to the body of the work redefine 
	\[
	\vv\leftrightarrow\thetav,\;\;\thetav\rightarrow\Pi\vv,\;\;\etav\leftrightarrow\left(I-\Pi\right)\vv
	\]
	In the proof one relies on local linear approximation of the quasi log-likelihood following with dominating probability from bound on stochastic component (theorem $\ref{BoundOnStochasticPart}$) and \Lnullc. For a quadratic form of parameters $\vvp$ and $\vvp_1$: $\mathbb{L}(\vvp,\vvp_1)=(\vvp-\vvp_1)^T\nabla L(\vvp_1)-\frac{\|\Dpn(\vvp-\vvp_1)\|^2}{2}$ introduce residual on the set $\Upsilon(r_0)$
	\[
	\alpha(\vvp_1,\vvp_2)=L(\vvp_1,\vvp_2)-\mathbb{L}(\vvp_1,\vvp_2).
	\]
	Then from the inequality 
	\[
	\|\D_{0}^{\prime-1}\nabla\E L(\vvp,\svvp)+\Dpn(\vvp-\svvp)\|\leq\delta(r_0)r_0,
	\]
	directly following from \Lnullc and theorem $\ref{BoundOnStochasticPart}$  one concludes  for $\vvp\in\Upsilon(r_0)$
	\[
	\|\D_0^{\prime-1}\nabla\alpha(\vvp,\svvp)\|\leq\diamondsuit(r_0,x),
	\]
	with the notation $\diamondsuit(r)=(\delta(r)+6\nu_0\mathfrak{Z}(x)\omega)r$. Triangle inequality for $\vvp_1,\vvp_2\in\Upsilon(r_0)$ gives
	\[
	\|\D_0^{\prime-1}\nabla\alpha(\vvp_1,\vvp_2)\|\leq2\diamondsuit(r_0,x).
	\]
	and it is evident that
	\[
	|\sqrt{2L(\vvp_1,\vvp_2)}-\sqrt{-2\mathbb{L}(\vvp_1,\vvp_2)}|\sqrt{-2\mathbb{L}(\vvp_1,\vvp_2)}\leq4\|\mathcal{D}_0(\vvp_1-\vvp_2)\|\diamondsuit(r_0,x),
	\]
	fro points $\vvp_1,\vvp_2$ such that $L(\vvp_1,\vvp_2)\geq0$. Moving forward consider transformation matrices 
	\[
	K\eqdef\Matrix{1}{-\D_{\thetav}^{-1}\D_{\thetav,\etav}}{-\D_{\etav}^{-1}\D_{\etav,\thetav}}{1}
	\]
	and
	\[
	K_1\eqdef\Matrix{1}{-\D_{\thetav,\etav}\D_{\etav}^{-1}}{\D_{\etav,\thetav}\D_{\thetav}^{-1}}{1},
	\] 
	then it can be seen that 
	\[
	\Dh\eqdef\Matrix{\Dh_0}{0}{0}{\Dh_1}=\Dpn K,
	\]
	and furthermore 
	\[
	\D_0^{\prime-1}=K\Dh^{-1}=\Dh^{-1}K_1.
	\]
	The transformation helps to get rid of non-diagonal entries of matrix $\Dpn$ and shape the form of the score $\xiv^s$. 
	
	Using proven above inequality under the alternative one has
	\[
	\left|\sqrt{2T_{LR}}-\|\Dh(\tvv^{\prime\prime}-\vv^{\prime\prime}_1)+\bias\|\right|\leq4\diamondsuit(r_0,x),
	\]
	where $\tvv^{\prime\prime}$ and $\vv^{\prime\prime}_1$ are such that $\vvp\eqdef K\vv^{\prime\prime}$. The fact that norm of truncated score vector is less then norm of full vector and local expansion for $\D_0^{-1}\nabla\alpha$ yields 
	\[
	\|\Dh_0\Svec{\widetilde{\thetav}^{\prime\prime}-\thetav_1^{\prime\prime}}{0}\|\leq2\diamondsuit(r_0,x),
	\] 
	and
	\[
	\|\Dh_0\Svec{0}{\widetilde{\etav}^{\prime\prime}-\etav_1^{\prime\prime}}-\xivh\|\leq\diamondsuit(r_0,x).
	\]
	Combining these three suffice the announced statement.
\end{proof}
In the bootstrap world an almost complete analogue of the theorem is attainable. It is evident that it takes place since we show that exactly similar conditions as in real world are replicated in the bootstrap world. 
\begin{theorem}\label{BSQW}
	Let conditions ($\ref{FSC}$) hold then with probability $\geq1-Ce^{-x}$
	\[
	\left|\sqrt{2T_{BLR}}-\|\xiv^{s}_{\sbt}\|\right|\leq7\diamondsuitb(r_0,x),
	\]
	where $\diamondsuitb(r,x)$ is given by
	\[
	\diamondsuitb(r,x)\eqdef\diamondsuit(r,x)+6\nu_0\omega_1(r)\mathfrak{Z}(x)r.
	\]
\end{theorem}
Let us specify the proof of this fact.
\begin{myproof}
	The underlying in the previous proof result - local linear approximation of a gradient - is sufficient. We are aiming thus at establishing - 
	\[
	\|\D_0^{-1}\nabla\alphab(\vv,\svv)\|\leq\diamondsuitb(r_0,x).
	\]
	It is easy to note that
	\[
	\|\D_0^{-1}\nabla\alphab(\vv,\svv)\|\leq\|\D_0^{-1}\nabla\zetab(\vv,\svv)\|+\|\D_0^{-1}\nabla\alpha(\vv,\svv)\|\leq6\nu_0\omega_1(r)r\mathfrak{Z}(x)+\diamondsuit(r,x),
	\]
	which follows from the theorem $\ref{DeviationOnBootstrapScore}$ and the previous proof. Therefore, for bootstrap world square root Wilks result is true with the same notations and with a minor change for $\bias^{\sbt}\equiv0$ since the hypothesis is exact and $\diamondsuit\rightarrow\diamondsuitb$.
\end{myproof}
\subsection{Matrix Inequalities}\label{Matrix}
In the section concentration of the operator norm of a random matrix - 
\[
\|S\|_{\infty}\eqdef\sup_{||\uv||=1,\uv\in\R^p} |\uv^T S\uv|,
\]
with an additive structure $S\eqdef\suml_{i=1}^nS_i$ is considered.  

The derivations generally follow techniques from Joel Tropp 2012 \cite{Tropp2012}, supported by analysis of operator functions from works Hansen, Pedersen \cite{hansen2003jensen}, Effors 2008 \cite{effros2009matrix} and Tropp \cite{tropp2012joint}. The exposition is self-contained and the chapter contains required prerequisites for the final result.

The main ingredient in the subsequent theory is concavity of the operator function
\[A\rightarrow tr\{\exp(H+\log A)\}\]
with respect to ordering on a positive-definite cone with $H$ being fixed self-adjoint operator. This fact can be found in the paper by Lieb 1973 \cite{lieb1973convex}. This chapter, however, follows mostly more direct argument of Joel Tropp 2012 \cite{tropp2012joint} exploiting joint convexity of relative entropy function.
\begin{theorem}(Lieb, 1973)\label{OperatorConcavity}
	For the fixed self-adjoint matrix $H$ function 
	\[A\rightarrow tr\{\exp(H+\log A)\}\]
	is concave with respect to positive-definite cone.
\end{theorem}

\subsubsection{Concavity theorem of Leib}
The proof of corollary from Leib's concavity theorem (theorem $\ref{OperatorConcavity}$) requires several supporting lemmas. Generalization of the Jensen inequality for operator functions is important, however the core constructive point in the proof is operator convexity of entropy function. The observation allows to infer that relative entropy  as a perspective of entropy is jointly convex. In view of the fact subsequent text contains slightly abused notation for relative entropy so that it equals exactly to the perspective.
\begin{lemma}(Lowner-Heinz)\label{Entropy}
	Define operator function (Entropy) - $\phi_e(X)\eqdef X\log X$ and define (relative entropy) - $\phi_e(X,Y)\eqdef X\log X-X\log Y$. Where $X\in\R^{p\times p}$ lie in Hilbert space of positive definite operators $\mathbb{H}^{++}_p$. Then  
	\begin{enumerate}
		\item $\phi_e(X)$ - operator convex. Namely for any positive definite $X_1,X_2$ and $\lambda\in(0,1)$
		\[
		\phi_e(\lambda X_1+(1-\lambda)X_2)\leq \lambda\phi_e(X_1)+(1-\lambda)\phi_e(X_2)
		\]
		\item $\phi_e(X,Y)$ - jointly operator convex. For any positive definite $X_1,Y_1,X_2,Y_2$ and $\lambda\in(0,1)$ holds
		\[
		\phi_e(\lambda X_1+(1-\lambda)Y_1,\lambda X_2+(1-\lambda)Y_2)\leq \lambda\phi_e(X_1,Y_1)+(1-\lambda)\phi_e(X_2,Y_2)
		\]
	\end{enumerate}
\end{lemma}
\text{}\\
Generalization of the lemma can be found under the name - Lowner-Heinz theorem. Below is an adopted proof of the required statement.
\begin{myproof}
	Let us demonstrate that inverse function $f:\R^{++}\rightarrow\R^{++}$ s.t. $f(t)=t^{-1}$ is operator convex function. 
	\\
	It is evident from definition for any invertible matrix $A\in\R^{p\times p}$ that
	\[
	\lambda_{i}^{A^{-1}}=1/\lambda^{A}_{i},\;\;\;i=\overline{1,p}
	\]
	where $\lambda_i^{A}$ is i-th eigenvalue of matrix $A$. And, therefore, also
	\[
	\lambda^{(I+A)^{-1}}_i=1/\lambda^{I+A}_i=1/(1+\frac{1}{\lambda^{A^{-1}}}),\;\;\;i=\overline{1,p}
	\]	
	which will be useful next. Also it is worth notion that in view of continuity only middle point convexity needs to be shown for function $f(t)=t^{-1}$ being convex. 
	\\
	Therefore, convexity is implied by the inequality
	\[
	\frac{1}{2}X_1^{-1}+\frac{1}{2}X_2^{-1}-\left(\frac{X_1+X_2}{2}\right)^{-1}\succ0 
	\] 	
	with respect to positive definite cone. Using the fact that matrix $C\eqdef X_1^{-1/2}X_2X_1^{-1/2}\succ0$ is positive definite helps to rearrange terms to get -
	\[
	\frac{1}{2}I+\frac{1}{2}C^{-1}-\left(\frac{I+C}{2}\right)^{-1}\succ0.
	\]
	Multiplying from both sides left hand side of inequality with unit vectors from orthogonal basis of eigenvectors matrix $C$ and using relations for eigenvalues above the matrix inequality is reduced to the $p$ inequalities on real line 
	\[
	\frac{1+\lambda^{C^{-1}}_i}{2}-\left(\frac{2}{1+\frac{1}{\lambda^{C^{-1}}_i}}\right)\geq0,\;\;\;i=\overline{1,p}
	\] 
	which obviously hold representing difference between arithmetic and harmonic means. Therefore $f(t)=t^{-1}$ is operator convex. 
	\\
	Next step is to demonstrate that entropy function can be represented as a weighted sum of functions $t^{-1}$. For that purpose introduce an integral representation of power of a matrix $X$. It can be seen that 
	\[
	X^p=c_p\int_{0}^{\infty}t^p(\frac{1}{t}-\frac{1}{t+X})dt,
	\]
	for $p\in\left(0,1\right)$ and $c_p$ is a constant depending only on $p$. Also multiplying by $X$ we get
	\[
	X^{p}=c_p\int_{0}^{\infty}t^{p-1}(\frac{X}{t}+\frac{1}{t+X}-I)dt,
	\]
	which now converges in the interval $p\in(1,2)$. And adding here 
	\[
	X\log X\eqdef\lim_{p\rightarrow1}\frac{X^p-X}{p-1},
	\]
	is sufficient to see that entropy is operator convex function. It follows a representation which is convex as a sum with positive coefficients of a convex functions. 
	\\
	Now it is left to demonstrate that relative entropy as was defined $\phi_e(X,Y)=X\log X - X\log Y$ is jointly convex function. Joint convexity can be seen via Hansen-Pedersen inequality \cite{hansen2003jensen} and relative entropy being perspective of $\phi_e(X)$ -
	\[
	\phi_e(X,Y)=\phi_e(XY^{-1})Y.
	\]
	Hansen-Pedersen inequality states
	\[
	\phi_e(A^TX_1A+B^TX_2B)\leq A^T\phi_e(X_1)A+B^T\phi_e(X_2)B
	\]
	for $A,B$ s.t. $A^TA+B^TB=I$. Then for $X=\lambda X_1+(1-\lambda)X_2$ and $Y=\lambda Y_1+(1-\lambda)Y_2$ and matrices $A=\lambda^{1/2}Y^{-1/2}Y_1^{1/2}$ and $B=(1-\lambda)^{1/2}Y^{-1/2}Y_2^{1/2}$ we receive
	\[
	\phi_e(X,Y)=\phi_e(A^T\frac{X_1}{Y_1}A+B^T\frac{X_2}{Y_2}B)Y\leq A^T\phi_e(\frac{X_1}{Y_1})AY+B^T\phi_e(\frac{X_2}{Y_2})BY\leq \lambda\phi_e(X_1,Y_2)+(1-\lambda)\phi_e(X_2,Y_2),
	\]
	which ends the proof.
\end{myproof}
\text{}\\
Following article by Tropp 2012 \cite{tropp2012joint} let us rely on geometric properties of $\phi_e(X)$. Quantifying the approach let us use Bregman operator divergence for entropy function and try to built its affine approximation which is in turn by lemma $\ref{Entropy}$ gives inequality
\[
D_{\phi_e}(X,Y)\eqdef\phi_e(X)-\phi_e(Y)-(\nabla\phi(Y),X-Y)\geq0.
\]
Above Bregman divergence was defined - $D_{\phi_e}(X,Y)$, and it is easy to see that $D_{\phi_e}(X,Y)=0$ iff $X=Y$. Therefore,the following lemma can be concluded.
\begin{lemma}(Variational Formula for Trace)\label{Variational}
	Let Y be a positive definite matrix, then
	\[
	tr Y=\sup_{X>0} tr(X\log Y-X\log X+X)
	\]
\end{lemma}
\text{}\\
Informally argument is presented above and one can skip the rigorous formal proof below.
\begin{myproof}
	Obviously from $D_{\phi_e}(X,Y)\geq0$ follows inequality for trace of $trD_{\phi_e}(X,Y)\geq0$. Therefore,
	\[
	tr Y\geq tr(X\log Y-X\log X+X).
	\]
	But equality holds iff $X=Y$, then we conclude the statement of the lemma.
\end{myproof}
Operator concavity also helps to derive the following lemma.
\begin{lemma}\label{SupConvexity}
	Function $\sup_{X>0}g(X,Y)$ is concave if $g(X,Y)$ is jointly concave.
\end{lemma}
\begin{myproof}
	First, suggest existence of $X_1,X_2$ and $Y_1,Y_2$ s.t. they provide a partial maximum to function $g(X,Y)$. Namely define them as
	\[
	X_1,Y_1:\;\;\;\sup_Xg(X,Y_1)=g(X_1,Y_1),
	\]
	\[
	X_2,Y_2:\;\;\;\sup_Xg(X,Y_2)=g(X_2,Y_2).
	\]
	Then observe that the set of inequalities hold
	\[
	\sup_X g(X,\lambda Y_1+(1-\lambda)Y_2)\leq
	g(\lambda X_1+(1-\lambda)X_2,\lambda Y_1+(1-\lambda)Y_2)\leq\]
	\[
	\leq\lambda g(X_1,Y_1)+(1-\lambda)g(X_2,Y_2)= \lambda\sup_Xg(X,Y_1)+(1-\lambda)\sup_Xg(X,Y_2)
	\]
	where joint operator convexity of $g$ was used along with the definition of points $(X_1,Y_1)$ and $(X_2,Y_2)$.
\end{myproof}
\text{}\\
Now we are in position to provide the proof of the theorem $\ref{OperatorConcavity}$ of Lieb, 1973.
\begin{myproof} (Theorem $\ref{OperatorConcavity}$).
	\\
	Let us start with variational formula from Lemma $\ref{Variational}$ for trace function saying
	\[
	tr Y = \sup_{X>0}tr(-\phi(X,Y)+X).
	\]
	Also from Lemma $\ref{Entropy}$ we know that $\phi(X,Y)$ is jointly convex and, therefore the trace of it is also jointly convex and, thus, supremum of the trace function is convex according to Lemma $\ref{SupConvexity}$. Now to demonstrate final result it suffice to substitute $Y$ with matrix $\exp\{H+\log A\}$ in variational formula for trace giving
	\[
	tr \exp\{H+\log A\}=sup_{X>0} tr\{ -\phi(X,A)-XH+X\}
	\] 
	and finally providing the advertised statement.
\end{myproof} 
\subsubsection{Master Bound}
Compared to the use of Golden-Thompson inequality
\[
tr\exp\{X+Y\}\leq tr\exp\{X\}\exp\{Y\},
\]
suitable for iid case one can follow theorem $\ref{OperatorConcavity}$ and improve upper bounds on tail distribution of operator norm of the random matrix. This improvement is inherent to the study of independent but not identically distributed random variables. 
\\
Obvious corollary from theorem $\ref{OperatorConcavity}$ can be useful further applications.
\begin{corollary}\label{AverageMIneqaulity}
	For any probability measure $\P$ and set of independent random matrices $\{S_i,i=\overline{1,n}\}$ holds 
	\[
	\E tr\exp\{\suml_{i=1}^nS_i\}\leq tr\exp\{\suml_{i=1}^n\log\E_i\exp{S_i}\}
	\]
\end{corollary}
\begin{myproof}
	Product structure of probability measure composed from independent marginal parts - $\P\eqdef\prod_{i=1}^{n}\P^{i}$ allows to write
	\[
	\E tr\exp\{\suml_{i=1}^nS_i\}=\E_1\E_2...\E_n tr\exp\{\suml_{i=1}^{n-1}S_i+\log\exp S_i\}.
	\]
	Using theorem $\ref{OperatorConcavity}$ helps to arrive at
	\[
	\E_1\E_2...\E_n tr\exp\{\suml_{i=1}^{n-1}S_i+\log\exp S_i\}\leq\E_1\E_2...\E_{n-1} tr\exp\{\suml_{i=1}^{n-1}S_i+\log\E_n\exp S_i\}.
	\]
	Iterating $n$-times the inequality accounting for the independence of $\{S_i\}$  finally relates
	\[
	\E tr\exp\{\suml_{i=1}^nS_i\}\leq tr\exp\{\suml_{i=1}^n\log\E_i\exp{S_i}\}
	\]
\end{myproof}
\text{}\\ 
This result can be easily combined with Markov exponential inequality to receive a bound on operator norm's tail probability.
\begin{theorem}\label{MasterInequality}(Master Bound)
	Suppose $\{S_i\in\R^{p\times p},i=\overline{1,n}\}$ are independent and  let us denote $S=\suml_{i=1}^nS_i$. Then following bound hold
	\[
	\P(\|S\|_{\infty}\geq t)\leq 2\inf_{\theta>0}e^{-\theta t}tr\exp(\suml_{i=1}^n\log\E_i\exp \theta S_i),
	\]
	for $\theta>0$ and $\|S\|_{\infty}=\sup_{\|\uv\|_2=1}|\uv^TS\uv|$.
\end{theorem}
\begin{myproof}
	The theorem follows directly from corollary $\ref{AverageMIneqaulity}$ and Markov exponential inequality. Write
	\[
	\P(\|S\|_{\infty}\geq z)=\P(\lambda_{max}(S)\vee\lambda_{max}(-S))\leq
	\]
	\[
	\leq\P(\lambda_{max}(S))+\P(\lambda_{max}(-S)).
	\]	
	It will be evident from the subsequent derivation that it is enough to control one of the probabilities. Also the spectral mapping theorem allows to state $\forall i$
	\[
	\exp\{\theta\lambda_{max}(S)\}=\lambda_{max}\exp\{\theta S\}
	\]
	and combined with a trivial inequality $\lambda_{max}\exp\{S\}\leq tr\exp\{S\}$ gives 
	\[
	\P(\lambda_{max}(S)\geq t)=\P(\exp\{\theta\lambda_{max}(S)\}\geq \exp\{\theta t\})=\P(\lambda_{max}\exp\{\theta(S)\}\geq \exp\{\theta t\})\leq
	\]
	\[
	\leq\P(tr\exp\{\theta(S)\}\geq \exp\{\theta t\})\leq e^{-\theta t}\E tr\exp\{\theta S\}.
	\]
	Now, applying corollary $\ref{AverageMIneqaulity}$ to the sum $S=\suml_iS_i$ of independent matrices we achieve desired result
	\[
	\P(||S||_{\infty}\geq t)\leq 2\inf_{\theta>0} e^{-\theta t}\E tr\exp\{\theta S\}\leq 2\inf_{\theta>0} e^{-\theta t}tr\{\exp(\suml_{i=1}^n\log\E_i\exp \theta S_i)\}.
	\]
\end{myproof}
\text{}\\
The subject of next two chapters - where we derive Bernstein inequality for two types of conditions on matrices $S_i$ - is to bound exponential moment of each independent matrix $S_i$, amounting to the bound on $\suml_{i=1}^{n}\log\E_i\exp(\theta S_i)$.
\subsubsection{Bernstein inequality for uniformly bounded  matrices.}
The matrix version of Bernstein type inequality requires supporting lemma for uniformly bounded matrices $S_i$ in a sense that $\|S_i\|_{\infty}\leq R$ for some positive and universal constant $R$.
\begin{lemma}\label{MomentsForBernstein}
	Suppose that random matrices $S_i$ for $i=\overline{1,n}$ are such that for some positive number $R$ we can found $\|S_i\|_{\infty}\leq R$ then it holds
	\[
	\begin{cases}
	\log \E_i \exp\{\theta S_i\}\leq\E_iS_i^2\psi_2(\theta R)/R^2 &\mbox{if}\;\; \forall\theta>0,\\
	\log \E_i \exp\{\theta S_i\}\leq\frac{\theta^2\E_iS_i^2}{2(1-\frac{R\theta}{3})} &\mbox{if}\;\; 0<\theta<\frac{3}{R},
	\end{cases}
	\]
	where we denote by $\psi_2(u)\eqdef e^{u^2}-1$.
\end{lemma}
\begin{myproof}
	The proof is classic and relies on the following series of inequalities. Let us decompose the expectation of exponent 
	\[
	\E_i \exp\{\theta S_i\}=\E_i\left[ I_p+\theta S_i+\theta^2 S_i^2\left(\frac{I_p}{2!}+\frac{\theta S_i}{3!}+\frac{\theta^2S_i^2}{4!}+...\right)\right]\leq
	\]
	\[
	\leq\E_i\left[ I_p+\theta S_i+\theta^2 S_i^2\left(\frac{1}{2!}+\frac{\theta \|S_i\|_{\infty}}{3!}+\frac{\theta^2\|S_i\|_{\infty}^2}{4!}+...\right)\right]\leq
	\]
	\[
	\leq I_p+\theta^2\E_i S_i^2\left[\frac{\exp\{\theta\|S_i\|_{\infty}\}-1-\theta\|S_i\|_{\infty}}{\theta^2\|S_i\|_{\infty}^2}\right].
	\]
	To proceed further it is suffice to denote that function - $\left[\frac{\exp\{u\}-1-u}{u^2}\right]$ - is non-decreasing in its argument and, therefore, last inequality can be substituted with a bound - $I_p+\E_i S_i^2\left[\frac{\exp\{\theta R\}-1-\theta R}{R^2}\right]$. Making also contribution here from inequalities $e^x-x\leq e^{x^2}$ and $1+x\leq e^x$ we arrive at
	\[
	\E_i \exp\{\theta S_i\}\leq I_p+\E_i S_i^2\frac{\psi_2(\theta R)}{R^2}\leq \exp\{\frac{\psi_2(\theta R)\E_iS_i^2}{R^2}\}.
	\]
	This concludes the first part of our statement. However, it is useful sometimes to have more convenient expression to work with. In the fashion of sub-exponential random variables it is nice to derive result with leading term proportional to $\theta^2$ in the right hand side of inequalities. This can be easily seen if we estimate series - $\left(\frac{1}{2!}+\frac{\theta R}{3!}+\frac{\theta^2R^2}{4!}+...\right)$ - using inequality $k!\geq23^{k-2}$. Explicitly we have for $\theta\leq\frac{3}{R}$
	\[
	\left(\frac{1}{2!}+\frac{\theta R}{3!}+\frac{\theta^2R^2}{4!}+...\right)\leq\frac{1}{2}\left(\suml_{k=2}^{\infty}\frac{(\theta R)^{k-2}}{3^{k-2}}\right)=\frac{1}{2(1-\theta R/3)},
	\]
	finally justifying the second part of the lemma
	\[
	\E_i \exp\{\theta S_i\}\leq I_p+\theta^2\E_i S_i^2\left(\frac{1}{2!}+\frac{\theta \|S_i\|_{\infty}}{3!}+\frac{\theta^2\|S_i\|_{\infty}^2}{4!}+...\right)\leq
	\]
	\[
	\leq I_p+\frac{\theta^2\E_iS_i^2}{2(1-\theta R/3)}\leq\exp\frac{\theta^2\E_iS_i^2}{2(1-\theta R/3)}.
	\]
\end{myproof}
\text{}\\
Matrix Bernstein inequality is easy step now to accomplish. All essential tools to provide concentration bound for the norm of random matrix was derived above. In essence one have to have two facts - first is the lemma $\ref{MomentsForBernstein}$ and second is master bound from previous section (theorem $\ref{MasterInequality}$). Those two are sufficient to justify Bernstein inequality for matrices.
\begin{theorem}\label{MatrixBernsteinInequality1}
	Suppose that random matrix $S=\suml_{i=1}^nS_i
	$ is s.t $\forall i$ there exist positive number $R$ bounding above $\|S_i\|_{\infty}\leq R$. Also denote $\sigma^2\eqdef\|\suml_{i=1}^n\E_iS_i^2\|_{\infty}$ and $\psi_2(u)=e^{u^2}-1$ then it holds for $\theta_{opt}\eqdef\frac{4\sigma^2\psi_2(\theta_{opt}R)}{R^2t}$
	\[
	\P(\|S\|_{\infty}\geq t)\leq 2p\exp\{-\frac{4\sigma^2\psi_2(\theta_{opt}R)}{R^2}\}=2p\exp\{-\theta_{opt}t\},
	\]
	which incurs \\
	a. for $t<\frac{4\psi_2(R)\sigma^2}{R^2}\eqdef t_{max}^2$
	\[
	\P(\|S\|_{\infty}\geq t)\leq2p\exp\{-\frac{R^2t^2}{4\psi_2(R)\sigma^2}\}=2p\exp\{-(t/t_{max})^2\},
	\] 
	b.Bernstein inequality
	\[
	\P(\|S\|_{\infty}\geq t)\leq2p\exp\{-\frac{t^2}{2\sigma^2(1+Rt/3\sigma^2)}\}
	\]
\end{theorem}
\begin{myproof}
	Straightforwardly apply master bound and lemma $\ref{MomentsForBernstein}$ to get -
	\[
	\P(\|S\|_{\infty}\geq t)\leq 2\inf_{\theta>0}e^{-\theta t}tr\exp(\suml_{i=1}^n\log\E_i\exp \theta S_i)\leq2\inf_{\theta>0}e^{-\theta t}tr\exp\left(\suml_{i=1}^n\E_iS_i^2\psi_2(\theta R)/R^2 \right)\leq
	\]
	\[
	\leq2p\inf_{\theta>0}\exp\left(-\theta t+\sigma^2\psi_2(\theta R)/R^2 \right).
	\]
	Analogously for the second case in lemma $\ref{MomentsForBernstein}$ for $0<\theta<\frac{3}{R}$
	\[
	\P(\|S\|_{\infty}\geq t)\leq2p\inf_{\theta>0}\exp\left(-\theta t+\frac{\theta^2\sigma^2}{2(1-\theta R/3)} \right).
	\]
	And the most unwieldy thing here is to optimize over $\theta$. Let us first deal with upper inequality above, namely we try to choose $\theta$ in a way to receive almost Gaussian like type of behavior for tails. For that we introduce $\alpha\eqdef\frac{\theta}{t}-\frac{\psi_2(\theta R)\sigma^2}{t^2R^2}$. It is evident that if lover bound on $\inf_{\theta>0}\alpha(\theta)$ is established then an upper-bound on right hand side of the first inequality will follow 
	\[
	\exp\left(-\theta t+\sigma^2\psi_2(\theta R)/R^2 \right)=\exp\{-\alpha t^2\}\leq\exp\{-\inf_{\theta>0}\alpha(\theta)t^2\}.
	\] 
	To proceed we rearrange alpha in the following way
	\[
	\alpha=-\left(\sqrt{\frac{\sigma^2\psi_2(\theta R)}{\theta^2 R^2}}\frac{\theta}{t}-\sqrt{\frac{\theta^2 R^2}{4\sigma^2\psi_2(\theta R)}}\right)^2+\frac{\theta^2 R^2}{4\sigma^2\psi_2(\theta R)},
	\]
	and now choose $\theta_{opt}\eqdef\frac{4\sigma^2\psi_2(\theta_{opt} R)}{R^2t}$ to approximate optimal $\alpha$. Then we have $\alpha(\theta_{opt})=\frac{4\sigma^2\psi_2(\theta_{opt}R)}{R^2t^2}$ and finally tail behavior -
	\[
	\P(\|S\|_{\infty}\geq t)\leq2p\exp\{-\frac{4\sigma^2\psi_2(\theta_{opt}R)}{R^2}\}=2p\exp\{-\theta_{opt}t\}.
	\] 
	Now we can analyze in more details the last formula. For example, in the case $\theta_{opt}<1$ it is easily seen that $\psi_2(\theta_{opt}R)<\theta_{opt}^2\psi_2(R)$ and, therefore, $\theta_{opt}>\frac{R^2t}{4\psi_2(R)\sigma^2}$, which in view of $\psi_2(\theta_{opt}R)\geq\theta_{opt}^2R^2$ recovers Gaussian tail behavior
	\[
	\P(\|S\|_{\infty}\geq t)\leq2p\exp\{-\frac{R^2t^2}{2\psi_2(R)\sigma^2}\}.
	\]
	This is useful illustration that if $R\rightarrow0$ then obviously one gets more Gaussian like tail behavior
	
	Also Bernstein inequality can be recovered in a classical form. Following below statements are usually can be seen as an argument to the proof of Bernstein like inequality and were used previously in the proof of lemma $\ref{MomentsForBernstein}$. In words using Taylor decomposition with inequality $k!\geq23^{k-2}$ yield estimate for all $0<\theta<\frac{3}{R}$
	\[
	\psi_2(\theta_{opt}R)\leq\frac{\theta_{opt}^2R^2}{2(1-2\theta_{opt}R/3)}.
	\]  
	Once again from definition of optimal point we can see that $\theta_{opt}\geq\frac{t(1-2\theta_{opt}R/3)}{2\sigma^2}$ and $\theta_{opt}\geq\frac{t}{2\sigma^2(1+Rt/3\sigma^2)}$. It can be easily verified for new point $\theta^1_{opt}\eqdef\frac{t}{2\sigma^2(1+Rt/3\sigma^2)}$ that $\theta^1_{opt}R<3$ and, therefore, we receive identical to classical Bernstein result
	\[
	\P(\|S\|_{\infty}\geq t)\leq2p\exp\{-\frac{t^2}{2\sigma^2(1+tR/3\sigma^2)}\}.
	\]
	This finalizes the proof of the theorem. It is left to establish only Bernstein type inequality in a conventional way. For that purpose let us use the second part of lemma $\ref{MomentsForBernstein}$ which yields inequality
	\[
	\P(\|S\|_{\infty}\geq t)\leq2p\inf_{\theta>0}\exp\left(-\theta t+\frac{\theta^2\sigma^2}{2(1-\theta R/3)} \right).
	\]
	Instead of optimization choosing $\theta=\frac{t}{\sigma^2(1+\frac{tR}{3\sigma^2})}$ we arrive at
	\[
	\P(\|S\|_{\infty}\geq t)\leq 2p\exp\left(-\frac{t^2}{2\sigma^2(1+\frac{tR}{3\sigma^2})} \right)
	\] 
	and finalize the second part of the theorem.
\end{myproof}
\subsubsection{Bernstein ineqaulity for sub-gaussian matrices}
To develop the theory in the section let us explore a bound analogous to the previous, however, requiring only sub-Gaussian tail behavior of a norm of the random matrix $S=\suml_{i=1}^nS_i$. Analogous result can be found in the work by Koltchinskii \cite{koltchinskii2013remark}.

Define for convex function $\psi_{\alpha}(u)\eqdef e^{u^\alpha}-1$ (see van der Vaart and Wellner \cite{van1996weak}) and operator norm $\|S_i\|_{op}$ a moment
\[
\|S_i\|_{\infty}^{\psi_\alpha}\eqdef\E_i\exp\{\|S_i\|^{\alpha}_{op}\}-1.
\]
If we bound this distance we will get Gaussian like behavior for tails and thus can complement our earlier discussion with more soft bound for the tail probability. I essence we can state
\begin{theorem}\label{MatrixBernstein2}
	Suppose that random matrix $S=\suml_{i=1}^nS_i\in \R^{p\times p}
	$ is s.t $\forall i$ there exist two positive numbers $C_n>\theta$ and $C_{p}>0$ for which 
	\[
	\|\theta S_i\|^{\psi_1}_{\infty}\leq C_{p}.
	\] 
	And choose $R$ and $\delta$ to satisfy 
	\[
	\frac{\delta\psi_2(3)\psi_1(3)}{R^3}=\frac{1}{\sigma^2}\suml_{i=1}^n\|6S_i/R\|^{\psi_1}_{\infty}.
	\]
	Then Bernstein matrix inequality holds again  
	\[
	\P(\|S\|_{op}\geq t)\leq2p\exp\{-\frac{t^2}{2\sigma^2(1+\delta)(1+Rt/3\sigma^2)}\},
	\]
	where $\sigma^2\eqdef\|\suml_i\E_iS_i^2\|_{op}$.
\end{theorem}
\begin{myproof}
	Let us start with the bound for exponential moments analogous to the ones in lemma $\ref{MomentsForBernstein}$. One can see for some positive constant $R$
	\[
	\E_i \exp\{\theta S_i\}\leq I_p+\frac{\E_iS_i^2\psi_2(\theta R)}{R^2}+\E_iS_i^2\frac{\psi_1(\theta\|S_i\|_{\infty})}{\|S_i\|^2_{\infty}}\boldsymbol{1}(\|S_i\|_{\infty}>R).
	\]
	The derivation remains the same as in the theorem $\ref{MatrixBernsteinInequality1}$ if the term is bounded
	\[
	\E_iS_i^2\frac{\psi_1(\theta\|S_i\|_{\infty})}{\|S_i\|^2_{\infty}}\boldsymbol{1}(\|S_i\|_{\infty}>R)
	\]
	with the goal to establish $R$ as small as possible such that it further sharpens bound on the quadratic term above according to the results from theorem $\ref{MatrixBernsteinInequality1}$. However, it is also important to keep reminder term with indicator small or at least proportional to the quadratic one, which naturally requires larger values of $R$. Resolving this trade off one comes at an optimal value $R$. 
	
	Proceed with substitution of indicator function with smooth approximation
	\[
	\boldsymbol{1}(\|S_i\|_{\infty}>R)\leq\frac{\psi_1(\theta\|S_i\|_{\infty}) R}{\psi_1(\theta R)\|S_i\|_{\infty}},
	\]
	where it was used that $\psi_1(u)/u$ is non-decreasing function. Thus, it leads to
	\[
	\E_iS_i^2\frac{\psi_1(\theta\|S_i\|_{\infty})}{\|S_i\|^2_{\infty}}\boldsymbol{1}(\|S_i\|_{\infty}>R)\leq\E_iS_i^2\frac{\psi_1^2(\theta\|S_i\|_{\infty})R}{\psi_1(\theta R)\|S_i\|_{\infty}^3}\leq
	\]
	\[
	\leq\frac{R}{\psi_1(\theta R)}\E_i\frac{\psi^2_1(\theta\|S_i\|_{\infty})}{\|S_i\|_{\infty}}I_p
	\]
	And to be consistent sum over $i$ of these terms needs to resemble quadratic one in the inequality above
	\[
	\frac{\delta\psi_2(\theta R)\psi_1(\theta R)}{R^{3}}= \frac{1}{\|\suml_{i=1}^n\E_iS_i^2\|_{op}}\suml_{i=1}^n\E_i\frac{\psi^2_1(\theta\|S_i\|_{\infty})}{\|S_i\|_{\infty}}.
	\] 
	Observe here that the function to the left is increasing and for sufficiently large values of $R$ and sufficiently small $\delta$ equality can be always satisfied. However, one additionally need to bound right hand side to demonstrate that such an $R$ exists. It can be done by following rough estimate for $\theta<C_n/2$
	\[
	\E_i\frac{\psi^2_1(\theta\|S_i\|_{\infty})}{\|S_i\|_{\infty}}\leq\E_i\psi_1(2\theta\|S_i\|_{\infty})\leq C_p.
	\]
	Although it is rough it provides enough evidence to justify existence of $R$. As to what value it equals exactly needs to be addressed implicitly via equality above. Since solution exists one can establish
	\[
	\suml_{i=1}^n\log\E_i \exp\{\theta S_i\}\leq (1+\delta)\frac{\|\suml_{i=1}^n\E_iS_i^2\|\psi_2(\theta R)}{R^2}
	\]
	and the first result follows from theorem $\ref{MatrixBernsteinInequality1}$.
	\\
	Let us dwell here finally on the constants $R$ and $\delta$. From the proof of theorem $\ref{MatrixBernsteinInequality1}$ we have $\theta R^*<3$. Then the definition of $R$ above helps tp built an upper estimate $R^{\prime}$ on it given by
	\[
	\frac{\delta\psi_2(3)\psi_1(3)}{R^{\prime3}}=\frac{1}{\sigma^2}\suml_{i=1}^n\E_i\psi_1(6\|S_i\|_{\infty}/R^{\prime}).
	\]
	and finalize the proof of theorem.
\end{myproof}
\text{}\\
Apply this result to specific case when matrices $S_i$ are built  based on sub-exponential random vectors $\xv_i\in\R^p$, for which we know that 
\[
\E_i\exp(\gammav\xv_i)\leq\exp\{||\gammav||_2^2/2n\},
\]
holds for any $i=\overline{1,n}$ and $\gammav\in\R^p$. Namely, define matrix $S_i$ as 
\[
S_i\eqdef\xv_i\xv_i^T-\E_i\xv_i\xv_i^T.
\]
we can draw from definition following inequality
\[
||S_i||_{\infty}\leq||\xv_i||^2_2+||\E_i\xv_i\xv_i^T||_{\infty}.
\]
In view of this note one can establish next corollary.
\begin{corollary}\label{Matrix_Brenstein_2_Constants}
	For matrices $S_i\eqdef\xv_i\xv_i^T-\E_i\xv_i\xv_i^T$, vectors $\xv_i$, for which exponential moment condition above holds with $n>2$, the constants from theorem $\ref{MatrixBernstein2}$ are
	\[
	R=\frac{12p}{n}
	\] 
	and there exist $0<\alpha<0.012$ such that
	\[
	\delta\leq \alpha\frac{p^3}{n^2\sigma^2}
	\]
\end{corollary}
\begin{myproof}
	For this technical proof one need to upper bound $\theta S_i$. And then applying definition of $R$ and $\delta$ from theorem $\ref{MatrixBernstein2}$ leads to the result. Notice that
	\[
	\|\theta S_i\|^{\psi_1}_{\infty}\eqdef\E_i\exp\{\theta \|S_i\|_{\infty}\}-1\leq
	e^{\theta\|\E_i\xv_i\xv_i^T\|_{\infty}}\E_i\exp\{\theta \|\xv_i\|^2_{2}\}-1.
	\]
	The expectation on right hand side of inequality can be explicitly calculated using exponential moment condition for $\xv_i$. It is evident that such an integral converges for $\theta<n/2$ and explicit calculation then gives
	\[
	\E_i\exp\{\theta \|\xv_i\|^2_{2}\}=\frac{1}{\sqrt[p]{2\pi}}\E_i\int_{\R^p}\exp\{\sqrt{2\theta}\xv_i\gammav-\frac{\|\gammav\|^2}{2}\}d\gammav\leq
	\]
	\[
	\leq\frac{1}{\sqrt[p]{2\pi}}\int_{\R^p}\exp\{\frac{(2\theta/n)\|\gammav\|^2}{2}-\frac{\|\gammav\|^2}{2}\}d\gammav
	\]
	which easily gives us
	\[
	\E_i\exp\{\theta \|\xv_i\|^2_{2}\}\leq(1-2\theta/n)^{-p/2}.
	\]
	and adapting it to $\theta S_i$ yields
	\[
	\|\theta S_i\|^{\psi_1}_{\infty}\leq e^{\theta\|\E_i\xv_i\xv_i^T\|_{\infty}}(1-2\theta/n)^{-p/2}-1.
	\]
	Now choose $R=12p/n$. Knowing that $\theta R<3$ we can check that here $2\theta< n/2$ as required for a norm to be finite. And using theorem $\ref{MatrixBernstein2}$ $\delta$ is given by the formula
	\[
	\delta=\frac{(12p)^3}{\psi_2(3)\psi_1(3)n^3\sigma^2}\suml_{i=1}^n\|n S_i/2p\|_{\psi_1}\leq
	\]
	\[
	\leq\frac{(12p)^3(e^{n\|\E_i\xv_i\xv_i^T\|_{\infty}/4p+1/2}\sqrt{p/(p-1)}-1)}{\psi_2(3)\psi_1(3)n^2\sigma^2}=\]
	\[
	=0.012\frac{p^3(e^{n\|\E_i\xv_i\xv_i^T\|_{\infty}/4p+1/2}\sqrt{p/(p-1)}-1)}{n^2\sigma^2}.
	\]
	If we further note $\|\E_i\xv_i\xv_i^T\|_{\infty}\leq C/n$, then using assumption on $n$ the order of $\delta$ is shown to be
	\[
	\delta\leq \alpha\frac{p^3}{n^3\sigma^2}\sim\frac{p^2}{n^2}
	\]
	for a positive constant satisfying $\alpha<0.012$. 
\end{myproof}
\text{}\\
This example of an empirical covariance matrix demonstrates sharp - in view of small parameter $\frac{p^3}{n}$ - bounds on constants $R$ and $\delta$, however not optimal ones. 
\subsection{Gaussian approximation}\label{GAR}
\subsubsection{Smooth representation of Kolmogorov distance.}\label{SRKD}

Introduce a smooth indicator function
\[
f(x)=\Ind(x>0)-\frac{1}{2}\textbf{sign}(x)e^{-\left|x\right|}
\]
and define a regular difference
\[
g_{\alpha}(t)\eqdef \E f\left(\alpha\|\xv_0\|^2-\alpha t\right)-\E f\left(\alpha\|\xv_1\|^2-\alpha t\right).
\]
One aims at studying the limiting object 
\[
g_{\infty}(t)\eqdef \E \Ind\left(\|\xv_0\|^2- t\right)-\E \Ind\left(\|\xv_1\|^2- t\right)=\P\left(\|\xv_0\|<t\right)-\P\left(\|\xv_1\|<t\right)
\]
the difference between multivariate probabilities. The smoothing function on the other hand allows for a structural characterization of the relation between $g_{\alpha}$ and $g_{\infty}$. 
\begin{lemma}\label{LinSmoothing}
	Assume that $g_{\alpha}(t)$ has smooth second derivative. Then it satisfies an ODE
	\[
	g_{\alpha}(t)=g_{\alpha}(t)+\frac{g^{\prime\prime}_{\alpha}(t)}{\alpha^2}.
	\]
	Moreover, an ordering holds
	\[
	\forall\alpha>0\;\;\sup_t\left|g_{\alpha}(t)\right|\leq\sup_t\left|g_{\alpha}(t)\right|.
	\]
\end{lemma}
\begin{proof}		 
	The kernel function $f$ admits an ODE representation
	\[
	\mathbb{L}_{\xv}\left( f(\alpha\|\xv\|^2-\alpha t)\right)=\mathbb{L}_{\xv}\left(\Ind\left(\|\xv\|^2>t\right)\right)+\frac{1}{\alpha^2}\left(\mathbb{L}_{\xv} \left( f(\alpha\|\xv\|^2-\alpha t)\right)\right)^{\prime\prime}_{t}
	\]
	with a linear integral operator $\mathbb{L}_{\xv}\left(\cdot\right)$ and an inequality
	\[
	\sup_t\left|\mathbb{L}_{\xv}\left( f(\alpha\|\xv\|^2-\alpha t)\right)\right|\leq\sup_t\abs{\mathbb{L}_{\xv}\left(\Ind\left(\|\xv\|^2>t\right)\right)}
	\]
	follows from the characterization of extreme points - second derivative in maximum is negative and positive in minimum. The same applies for the difference $g_{\alpha}(t)$.
\end{proof} 
A natural candidate for the investigation of an underlying structure of the problem is Fourier analysis as the ODE in the lemma [$\ref{LinSmoothing}$] resembles an oscillator with a complex $\alpha$. Thus, define a spectrum of $g_{\alpha}(t)$ and $g_{\infty}(t)$ as follows
\[
G_{\infty}(\omega)=\mathcal{F}\left(g_{\infty}(t)\right)=\int^{\infty}_{-\infty}g_{\infty}(t)e^{-i\omega t}dt
\]
\[
G_{\alpha}(\omega)=\mathcal{F}\left(g_{\alpha}(t)\right)=\int^{\infty}_{-\infty}g_{\alpha}(t)e^{-i\omega t}dt
\]
respectively and with the convention for $\omega$ being a frequency scaled by $2\pi$. Additionally we analytically extend the spectra on $\omega\in\mathbb{C}$ which is the derivation crucial in the inversion step and we elaborate on that later (see lemma $\ref{Inverse}$). 

Easy to notice that in the Fourier world the connection between $G_{\infty}(\omega)$ and $G_{\alpha}(\omega)$ is straightforward and given by 
\[
G_{\alpha}(\omega)=G_{\infty}(\omega)-\frac{\omega^2}{\alpha^2}G_{\alpha}(\omega)
\]
which yields
\begin{equation}\label{ComplexDiff}
G_{\alpha}(\omega)=\frac{\alpha^2G_{\infty}(\omega)}{(\alpha-i\omega)(\alpha+i\omega)}.
\end{equation}
The central observation for further analysis is that $\alpha$ can be taken as a complex number $\alpha\in\mathbb{C}$ in the ODE leaving equation $\ref{ComplexDiff}$ intact. 

Introduce supplementary clockwise oriented contours $\mathbf{S}(r),-i\mathbf{S}(r)$ in complex plane.
\[
\begin{tikzpicture}
\begin{scope}[shift={(0,0)}]
\draw [->](0,-2.2)--(0,2) node[right]{$\mathbf{Im}(z)$};
\draw [->](-1.5,0)--(2.2,0) node[right]{$\mathbf{Re}(z)$};
\draw [thick, ->] (0,0) (1.75,0) arc (0:-90:1.75cm);
\draw [thick, ->] (0,0) (0,1.75) arc (90:0:1.75cm);
\draw [black, thick,->]  (0,-1.75) node[anchor=east]{-r} -- (0,0);
\draw [black, thick,->]  (0,0) -- (0,1.75) node[anchor=east]{r};
\draw (0,-2.1) node[anchor=west]{$-i\mathbf{S}(r)$};
\end{scope}	
\begin{scope}[shift={(-5,0)}]
\draw [->](0,-1.2)--(0,2) node[right]{$\mathbf{Im}(z)$};
\draw [->](-2.2,0)--(2.2,0) node[anchor=west]{$\mathbf{Re}(z)$};
\draw [thick, ->]  (-1.75,0) arc (180:90:1.75cm);
\draw [thick, ->]  (0,1.75) arc (90:0:1.75cm) node[anchor=north]{-r} ;
\draw [black, thick,->]  (0,0) -- (-1.75,0) node[anchor=north]{r};
\draw [black, thick,->]  (1.75,0) -- (0,0);
\draw (0,-2.1) node[anchor=west]{$\mathbf{S}(r)$};
\end{scope}
\end{tikzpicture}
\]
One option to find the closed-form connection between $G_{\alpha}(\omega)$ and $G_{\infty}(\omega)$ independent from $\alpha$ is to integrate $G_{\alpha}(\omega)$ over $\alpha\in -i\mathbf{S}(r)$. The step gains an additional smoothness as we will see below. After inspecting the poles of $G_{\alpha}(\omega)$ on the picture
\[
\begin{tikzpicture}
\draw [->](0,-2.2)--(0,2) node[right]{$\mathbf{Im}(\alpha)$};
\draw [->](-2.2,0)--(2.2,0) node[right]{$\mathbf{Re}(\alpha)$};
\draw [thick, ->] (0,0) (1.75,0) arc (0:-90:1.75cm);
\draw [thick, ->] (0,0) (0,1.75) arc (90:0:1.75cm);
\draw [black, thick,->]  (0,-1.75) node[anchor=east]{-r} -- (0,0);
\draw [black, thick,->]  (0,0) -- (0,1.75) node[anchor=east]{r};
\draw [black, ->]  (0,0) -- (1.2,0.25) node[anchor=south]{$-i\omega$} node{\textbf{.}};
\draw [black, ->]  (0,0) -- (-1.2,-0.25) node[anchor=north]{$i\omega$} node{\textbf{.}};
\draw (0,-2.1) node[anchor=west]{$-i\mathbf{S}(r)$};
\end{tikzpicture}
\]
it is obvious in view of the Cauchy's residue theorem to conclude for the convolution
\begin{equation}\label{Convolution}
\frac{1}{i\pi}\int_{-i\mathbf{S}(r_0)}(-\alpha)^{k-1}G_{\alpha}(\omega)d\alpha=
\end{equation}
\[
=\frac{1}{i\pi}\int_{-i\mathbf{S}(r_0)}\frac{(-\alpha)^{k+1}G_{\infty}(\omega)}{(\alpha+i\omega)(\alpha-i\omega)}d\alpha=
\begin{cases}
\left(i\omega\right)^k G_{\infty}(\omega) & \;\; \omega\in \mathbf{S}(r_0),  \\
\left(-i\omega\right)^k G_{\infty}(\omega) &-\omega\in \mathbf{S}(r_0), \\
0  & \textbf{else}.
\end{cases}
\]
where we also multiplied the spectrum by additional $(-\alpha)^{k-1}$ to generalize and expand on the idea later (see corollary $\ref{CorollaryDensity}$).

The formula $\ref{Convolution}$ gives clear explanation how initial function $g_{\infty}(t)$ can be regularized through the $g_{\alpha}(t)$. The answer above suggests that convolution of our 'kernels' $g_{\alpha}(t)$ is equivalent to the differentiation. For now the connection is settled in the Fourier world and one need to translate the result back into the initial objects. For the purpose let us rewrite the Fourier inversion formula as an integration in a complex plane.
\begin{lemma}\label{Inverse}
	Assume continuous p.d.f. of $\|\xv_0\|^2$ and $\|\xv_1\|^2$, then the functions $g_{\alpha}(t)$ and $g_{\infty}(t)$ can be represented as
	\[
	g_{\alpha}(t)=\frac{1}{2\pi}\int_{\mathbf{S}(r_0)}G_{\alpha}(\omega)e^{i\omega t}d\omega
	\]
	and
	\[
	g_{\infty}(t)=\frac{1}{2\pi}\int_{\mathbf{S}(r_0)}G_{\infty}(\omega)e^{i\omega t}d\omega
	\]
	for $t>0$ and $r_0$ s.t. $\mathbf{S}(r_0)$ covers all the poles of a spectrum $G_{\infty}(\omega)$.
\end{lemma}
\begin{proof}
	Let us compute explicitly $G_{\infty}(\omega)$ to proceed -
	\[
	G_{\infty}(\omega)=\int^{\infty}_{-\infty}g_{\infty}(t)e^{-i\omega t}dt=\E\left[\mathcal{F}\left(\Ind(\xv_0\in\mathcal{B}_t)\right)-\mathcal{F}\left(\Ind(\xv_1\in\mathcal{B}_t)\right)\right]
	\]
	\[
	=\frac{\E e^{-i\omega\|\xv_0\|^2}-\E e^{-i\omega\|\xv_1\|^2}}{i\sqrt{2\pi}\omega}+\sqrt{\frac{\pi}{2}}\E e^{-i\omega\|\xv_0\|^2}\delta(\omega)-\E e^{-i\omega\|\xv_1\|^2}\delta(\omega)
	\]
	\[
	=\frac{\E e^{-i\omega\|\xv_0\|^2}-\E e^{-i\omega\|\xiv_1\|^2}}{i\omega\sqrt{2\pi}}.
	\]
	On the other hand the contour $\mathbf{S}(r)$ can be seen as a sum of the real-line and semicircle parts - $\mathbf{S}(r)=[-r,r]\cup\mathbf{Arc}(r)$ - where the latter conforms the limit
	\[
	\lim_{r\rightarrow\infty}\int_{\mathbf{Arc}(r)}G_{\infty}(\omega)e^{i\omega t}\dx{\omega}\leq \lim_{r\rightarrow\infty}\pi r \sup_{\omega\in\mathbf{Arc}(r)}|G_{\infty}(\omega)|\leq
	\]
	\[
	\leq\lim_{r\rightarrow\infty}\sup_{\omega\in\mathbf{Arc}(r)}\left|\E e^{-i\omega\|\xv_0\|^2}-\E e^{-i\omega\|\xv_1\|^2}\right|=0.
	\]
	Therefore, the inverse is given as an integral over $\mathbf{S}(\infty)$
	\[
	g_{\infty}(t)\eqdef\frac{1}{2\pi}\int_{-\infty}^{\infty}G_{\infty}(\omega)e^{i\omega t}\dx{\omega}+\frac{1}{2\pi}\int_{\mathbf{Arc}(\infty)}G_{\infty}(\omega)e^{i\omega t}\dx{\omega}
	\]
	\[
	=\frac{1}{2\pi}\int_{\mathbf{S}(\infty)}G_{\infty}(\omega)e^{i\omega t}\dx{\omega}.
	\]	
	Defining now the critical points of $G_{\infty}(\omega),G_{\alpha}(\omega)$ as $\{\omega_{j=\overline{1,n}}\}$ and $\{\omega_{j=\overline{1,n}},-i\alpha,i\alpha\}$ respectively (see the equation $\ref{ComplexDiff}$) we see that by the assumption of the lemma they are covered by the $\mathbf{S}(r_0)$.
	\[
	\begin{tikzpicture}
	\draw [->](0,-1.2)--(0,2) node[right]{$\mathbf{Im}(\omega)$};
	\draw [->](-2.2,0)--(2.2,0) node[anchor=west]{$\mathbf{Re}(\omega)$};
	\draw [thick, ->]  (-1.75,0) arc (180:90:1.75cm);
	\draw [thick, ->]  (0,1.75) arc (90:0:1.75cm);
	\draw [black, thick,->]  (-1.75,0) node[anchor=north]{-r} -- (0,0);
	\draw [black, thick,->]  (0,0) -- (1.75,0) node[anchor=north]{r};
	\draw (0,-0.5) node[anchor=west]{$\mathbf{S}(r_0)$};
	\draw (0,0.1) node{\textbf{.}};
	\draw (0,0.2) node{\textbf{.}};
	\draw (0,0.3) node{\textbf{.}} node[anchor=west]{$\omega_{j-1}$};
	\draw (0,0.4) node{\textbf{.}};
	\draw (0,0.5) node{\textbf{.}};
	\draw (0,0.6) node{\textbf{.}} node[anchor=west]{$\omega_j$};
	\draw (0,0.7) node{\textbf{.}};
	\draw (0,0.9) node{\textbf{.}} node[anchor=west]{$\omega_{j+1}$};
	\draw (0,1.2) node{\textbf{.}};
	\draw (0,1.6) node{\textbf{.}};
	\end{tikzpicture}
	\]
	Therefore, the Cauchy's residue theorem puts the equivalence
	\[
	g_{\infty}(t)=\frac{1}{2\pi}\int_{\mathbf{S}(\infty)}G_{\infty}(\omega)e^{i\omega t}\dx{\omega}=\frac{1}{2\pi}\int_{\mathbf{S}(r_0)}G_{\infty}(\omega)e^{i\omega t}\dx{\omega}
	\]
	and completes the argument.
\end{proof}
\begin{rmrk}
	Note that in the proof above we conclude from positiveness of $\|\xv_0\|^2,\|\xv_1\|^2$ that all the poles but for the $i\alpha$ lie above the real line.
\end{rmrk}
With these two facts - the inversion above and convolution over $\alpha$ - one comes to the concluding step of the section. From the $\alpha$ perspective the pole structure of the $G_{\infty}(\omega)$ looks like it is drawn on the picture below.
\[
\begin{tikzpicture}
\draw [->](0,-2.2)--(0,2) node[right]{$\mathbf{Im}(\alpha)$};
\draw [->](-2.2,0)--(2.2,0) node[right]{$\mathbf{Re}(\alpha)$};
\draw [thick, ->] (0,0) (1.75,0) arc (0:-90:1.75cm);
\draw [thick, ->] (0,0) (0,1.75) arc (90:0:1.75cm);
\draw [black, thick,->]  (0,-1.75) node[anchor=east]{-$r_0$} -- (0,0);
\draw [black, thick,->]  (0,0) -- (0,1.75) node[anchor=east]{$r_0$};
\draw (0.1,0) node{\textbf{.}};
\draw (0.2,0) node{\textbf{.}};
\draw (0.3,0) node{\textbf{.}};
\draw (0.4,0) node{\textbf{.}};
\draw (0.5,0) node{\textbf{.}};
\draw (0.6,0) node{\textbf{.}};
\draw (0.7,0) node{\textbf{.}};
\draw (0.9,0) node{\textbf{.}} node[anchor=north]{$-i\omega_j$};
\draw (1.2,0) node{\textbf{.}};
\draw (1.6,0) node{\textbf{.}};
\draw (0,-2.1) node[anchor=west]{$-i\mathbf{S}(r_0)$};
\end{tikzpicture}
\]
Where by the definition of $r_0$ the convolution over $\alpha$ preserves the pole structure of $g_{\infty}(t)$. Thus, the inversion from the lemma [$\ref{Inverse}$] allows to relate explicitly the function $g_{\infty}(t)$ and a part of the equation $\ref{Convolution}$ where $\omega\in\mathbf{S}(r_0)$. Merging the two one can state the theorem.
\begin{theorem}\label{MainTh}
	Assume $k$-continuous c.d.f. of $\|\xv_0\|^2$ and $\|\xv_1\|^2$, then it holds
	\[
	g^{\left(k\right)}_{\infty}(t)=(-1)^k2i\int_{-i\mathbf{S}(r_0)}\alpha^{k-1}g_{\alpha}(t)d\alpha.
	\]
	for $r_0$ s.t. $\mathbf{S}(r_0)$ covers all the poles of the spectrum $G_{\infty}(\omega)$.
\end{theorem}
\begin{myproof}
	Justified by the lemma [$\ref{Inverse}$] and using the equation [$\ref{Convolution}$] integrate over $\alpha$ 
	\[
	\int_{-i\mathbf{S}(r_0)}(-\alpha)^{k-1}g_{\alpha}(t)\dx{\alpha}\overset{L\;\ref{Inverse}}{=}\frac{1}{2\pi}\int_{-i\mathbf{S}(r_0)}(-\alpha)^{k-1}\int_{\mathbf{S}(r_0)}G_{\alpha}(\omega)e^{i\omega t}\dx{\omega}\dx{\alpha}
	\]
	\[
	=\frac{1}{2\pi}\int_{\mathbf{S}(r_0)} i\pi \left(i\omega\right)^k G_{\infty}(\omega)e^{i\omega t}\dx{\omega}=\frac{i}{2}\left(\int_{\mathbf{S}(r_0)}G_{\infty}(\omega)e^{i\omega t}\dx{\omega}\right)^{\left(k\right)}_{t}\overset{L\;\ref{Inverse}}{=}-\frac{g^{\left(k\right)}_{\infty}(t)}{2i}.
	\]
	The answer concludes the proof.
\end{myproof}
The theorem in turn amounts to the corollaries.
\begin{corollary}\label{CorollaryLc}
	Introduce a function 
	\[
	h_{\alpha}(\xv,t)=\max\{\|\xv\|^2-t,0\}+\frac{1}{2\alpha}e^{-\alpha\left|\|\xv\|^2-t\right|}
	\]
	and an integral operator
	\[
	\mathcal{A}\left(\cdot\right)\eqdef2i\int_{-i\mathbf{S}(r_0)}\E\left(\cdot\right)d\alpha,
	\]
	then under assumptions in the theorem [$\ref{MainTh}$] we have for $k=1$
	\[
	\P\left(\|\xv_0\|^2<t\right)-\P\left(\|\xv_1\|^2<t\right)=\mathcal{A}\left(h_{\alpha}(\xv_1,t)-h_{\alpha}(\xv_0,t)\right) .
	\]
\end{corollary}
\begin{corollary}\label{CorollaryDensity}
	Introduce a function 
	\[
	h_{\alpha}(\xv_1,\xv_0,t) = \int_{0}^t\left[ f(\alpha\|\xv_1\|^2-\alpha x) -  f(\alpha\|\xv_0\|^2-\alpha x)\right] dx
	\]
	and an integral operator
	\[
	\mathcal{B}\left(\cdot\right)\eqdef 2i\int_{-i\mathbf{S}(r_0)}\E\left(\cdot\right)\alpha d\alpha,
	\]
	then under assumptions in the theorem [$\ref{MainTh}$] we have for densities of $\|\xv_1\|^2$ and $\|\xv_0\|^2$
	\[
	\rho_{\|\xv_1\|^2}(t)-\rho_{\|\xv_0\|^2}(t)=\mathcal{B}h_{\alpha}(\xv_1,\xv_0,t).
	\]
\end{corollary}
In essence the theorem $\ref{MainTh}$ and corollary $\ref{CorollaryLc}$ allow for the direct application of a simple Taylor expansion to the function $h_{\alpha}$. Namely the statements claim that one can differentiate under the operator $\mathcal{A}$. It is a subject of the next chapter to explore the use case of Gaussian approximation.

\subsubsection{GAR on Euclidean balls.}\label{GARsection}

The road-map of the following application case of the section above is to use purely Taylor decomposition up to the third term. Aside from the fact the proof is technical and presents no specific interest except for the outcome, which is comparable to the work of Betnkus 2005 \cite{bentkus2005lyapunov}.

Classic Lindenberg construction entails the following framework -
\begin{itemize}
	\item Define vectors $\xiv_1\eqdef\suml_{i=1}^{n}\xiv_{1,i}$ and $\xiv_2\eqdef\suml_{i=1}^{n}\xiv_{2,i}$ s.t. 
	\begin{enumerate}
		\item $\xiv_{1,i}$ and $\xiv_{2,i}$ are independent mutually and over $i=\overline{1,n}$ 
		\item $\E\xiv_{1,i}\xiv_{1,i}^T=\E\xiv_{2,i}\xiv_{2,i}^T=\Sigma/n$
	\end{enumerate}
	\item and design the chain vector  
	\[
	\xiv_{/i}\eqdef\suml_{j=0}^{i-1}\xiv_{1,j}+\suml_{j=i+1}^{n+1}\xiv_{2,j}
	\]
	with a convention $\xiv_{1,0}=\xiv_{2,n+1}=0$.
\end{itemize}
With the introduced notations one claims.
\begin{theorem}\label{S1GAR}
	Assume continuous measures and the framework above, then it holds
	\[
	\sup_{t}\left|\P\left(\|\xiv_1\|<t\right)-\P\left(\|\xiv_2\|<t\right)\right|\leq C_{\P}\left(Tr\Sigma\right)^{3/2}\frac{r_0^4}{\sqrt{n}}
	\]
	with the universal constant 
	\[
	C_{\P}\eqdef 48\sup_{i=\overline{1,n},\|\gammav\|^2=1} \left[\E_i\left|\gammav^T\xiv_{/i}+\gammav^T\av\right|^3+\frac{1}{r_0}\E_{/i}\left|\gammav^T\xiv_{/i}+\gammav^T\av\right|\right]
	\]
	and a distribution dependent only vector $\av$.
\end{theorem}
\begin{myproof}
	Let us start the proof forming a chain
	\[
	g_{\infty}(t)=\suml_{i=1}^n\E_ig^i_{\infty}(t)=
	\]
	\[
	=\suml_{i=1}^{n} \E_i\left[\P_{/i}\left(\|\xiv_{/i}+\xiv_{1,i}\|^2> t\right)-\P_{/i}\left(\|\xiv_{/i}+\xiv_{2,i}\|^2> t\right)\right],
	\]
	where the latter chaining sum is called the Lindenberg device. The main objective of the device is to exploit independence of $\xiv_{/i}$, $\xiv_{1,i}$ and $\xiv_{2,i}$ amounting to the cancellation of first and second terms in Taylor expansion. 
	
	However to exploit Taylor decomposition we first define smooth counterpart of the Lindenberg summand $g^i_{\infty}(t)$ as per the corollary $\ref{CorollaryLc}$
	\[
	h^i_{\alpha}(\xiv_{1,i},\xiv_{2,i},t)\eqdef\E_{/i}\left[\int_{0}^tf\left(\alpha\|\xiv_{/i}+\xiv_{1,i}\|^2-\alpha x\right)dx-\int_{0}^tf\left(\alpha\|\xiv_{/i}+\xiv_{2,i}\|^2-\alpha x\right)dx\right].
	\]
	From the theorem $\ref{MainTh}$, the contour integral over $|\alpha|<r_0$ recovers the limiting difference
	\[
	\P\left(\|\xiv_{/i}+\xiv_{1,i}\|^2> t\right)-\P\left(\|\xiv_{/i}+\xiv_{2,i}\|^2> t\right)=\mathcal{A}h^i_{\alpha}(\xiv_{1,i},\xiv_{2,i},t)
	\]
	\[
	=2i\int_{i\mathbf{S}(r_0)}\E_ih^i_{\alpha}(\xiv_{1,i},\xiv_{2,i},t)d\alpha
	\]
	for some fixed positive constant $r_0$. On the other hand assuming continuous pdf the function
	\[
	f_{\alpha}(\xv,t)\eqdef \E_{/i}\int_{0}^tf\left(\alpha\|\xiv_{/i}+\xv\|^2-\alpha x\right)dx
	\]
	is tree times continuously differentiable with respect to $\xv$  and admits Taylor expansion up to the third term. 
	
	Computing derivatives one has
	\begin{itemize}
		\item $\frac{\partial f_{\alpha}(0,t)}{\partial \xv}=2\alpha\E_{/i}f\left(\alpha\|\xiv_{/i}\|^2-\alpha t\right)\xiv_{/i}$
		\item $\frac{\partial^2 f_{\alpha}(0,t)}{\partial\xv\partial\xv}=4\alpha^2\E_{/i}f^{\prime}\left(\alpha\|\xiv_{/i}\|^2-\alpha t\right)\xiv_{/i}\xiv_{/i}^T+2\alpha\E_{/i}f\left(\alpha\|\xiv_{/i}\|^2-\alpha t\right)I_p$
		\item $\frac{\partial^3 f_{\alpha}(\av,t)}{\partial\xv\partial\xv\partial\xv}=8\alpha^3\E_{/i}f^{\prime\prime}\left(\alpha\|\xiv_{/i}+\av\|^2-\alpha t\right)\left(\xiv_{/i}+\av\right)\otimes\left(\xiv_{/i}+\av\right)\otimes\left(\xiv_{/i}+\av\right)$
		\[
		+8\alpha^2\E_{/i}f^{\prime}\left(\alpha\|\xiv_{/i}+\av\|^2-\alpha t\right)\left(\xiv_{/i}+\av\right)\otimes I_p
		\]
	\end{itemize}
	and, therefore, we obtain for the zero, first and second order approximation of the difference $\E_i\left[f_{\alpha}(\xiv_{1,i},t)-f_{\alpha}(\xiv_{2,i},t)\right]$ 
	\begin{itemize}
		\item $\E_i\left[f_{\alpha}(0,t)-f_{\alpha}(0,t)\right]=0$
		\item $\E_i\left[\frac{\partial^T f_{\alpha}(0,t)}{\partial\xiv_{1,i}}\xiv_{1,i}-\frac{\partial^T f_{\alpha}(0,t)}{\partial\xiv_{2,i}}\xiv_{2,i}\right]=0$
		\item $\E_i\left[\xiv_{1,i}^T\frac{\partial^2 f_{\alpha}(0,t)}{\partial\xiv_{1,i}\partial\xiv_{1,i}}\xiv_{1,i}-\xiv_{2,i}^T\frac{\partial^2 f_{\alpha}(0,t)}{\partial\xiv_{2,i}\partial\xiv_{2,i}}\xiv_{2,i}\right]=0$
	\end{itemize}
	by the designed independence of the $\xiv_{/i},\xiv_{1,i}$ and $\xiv_{2,i}$. 
	
	The Taylor expansion of $\E_ih^i_{\alpha}(\xiv_{1,i},\xiv_{2,i},t)$ then reads as
	\[
	\E_ih^i_{\alpha}(\xiv_{1,i},\xiv_{2,i},t)=\frac{1}{6}\E_i\left[\frac{\partial^3 f_{\alpha}(\av_i,t)}{\partial\xiv_{1,i}\partial\xiv_{1,i}\partial\xiv_{1,i}}\cdot\xiv_{1,i}\otimes\xiv_{1,i}\otimes\xiv_{1,i}-\frac{\partial^3 f_{\alpha}(\bv_i,t)}{\partial\xiv_{2,i}\partial\xiv_{2,i}\partial\xiv_{2,i}}\cdot\xiv_{2,i}\otimes\xiv_{2,i}\otimes\xiv_{2,i}\right]
	\]
	for some vectors $\av_i,\bv_i$ depending implicitly on $\xiv_{1,i}$ and $\xiv_{2,i}$ respectively. The difference is obviously bounded by
	\[
	\E_ih^i_{\alpha}(\xiv_{1,i},\xiv_{2,i},t)\leq\frac{1}{6}\E_i\left|\frac{\partial^3 f_{\alpha}(\av_i,t)}{\partial\xiv_i\partial\xiv_i\partial\xiv_i}\right|\|\xiv_{1,i}\|^3+\frac{1}{6}\E_i\left|\frac{\partial^3 f_{\alpha}(\bv_i,t)}{\partial\txiv_i\partial\txiv_i\partial\txiv_i}\right|\|\xiv_{2,i}\|^3
	\]
	and it holds
	\[
	\E_ih^i_{\alpha}(\xiv_{1,i},\xiv_{2,i},t)\leq\frac{1}{6}\sup_{\|\gammav\|^2=1}\left|\frac{\partial^3 f_{\alpha}(\av^*,t)}{\partial\xiv_i\partial\xiv_i\partial\xiv_i}\gammav\otimes\gammav\otimes\gammav\right|\E_i\|\xiv_{1,i}\|^3+\frac{1}{6}\sup_{\|\gammav\|^2=1}\left|\frac{\partial^3 f_{\alpha}(\bv^*,t)}{\partial\txiv_i\partial\txiv_i\partial\txiv_i}\gammav\otimes\gammav\otimes\gammav\right|\E_i\|\xiv_{2,i}\|^3
	\]
	for some constant vectors $\av^*,\bv^*$. But additionally we know for a fixed $\av$
	\[
	\left|\frac{\partial^3 f_{\alpha}(\av,t)}{\partial\xv\partial\xv\partial\xv}\right|=\Big|4\alpha^3\E_{/i}f^{\prime\prime}\left(\alpha\|\xiv_{/i}+\av\|^2-\alpha t\right)\left(\gammav^T\xiv_{/i}+\gammav^T\av\right)^3+ 
	\]
	\[
	+4\alpha^2\E_{/i}f^{\prime}\left(\alpha\|\xiv_{/i}+\av\|^2-\alpha t\right)\left(\gammav^T\xiv_{/i}+\gammav^T\av\right)\Big|\leq 
	\]
	\[
	\leq 4 r_0^3\E_i\left|\gammav^T\xiv_{/i}+\gammav^T\av\right|^3+4r_0^2\E_{/i}\left|\gammav^T\xiv_{/i}+\gammav^T\av\right|
	\]
	by $f^{\prime}\leq0.5$ and $f^{\prime\prime}\leq0.5$ and thus
	\[
	\int_{i\mathbf{S}(r_0)}\E_ih^i_{\alpha}(\xiv_{1,i},\xiv_{2,i},t)d\alpha\leq\frac{C_{\P}r_0^4}{4} (\E_i\|\xiv_{2,i}\|^3+\E_i\|\xiv_{2,i}\|^3)
	\]
	Therefore, using the bound on the moments from lemma $\ref{Moments}$ in the appendix and summing over Lindenberg chain, we come at the inquired statement
	\[
	\sup_t\left|\P(\|\xiv_1\|^2>t)-\P(\|\xiv_2\|^2>t)\right|\leq C_{\P}\left(Tr\Sigma\right)^{3/2}\frac{r_0^4}{\sqrt{n}}.
	\]	
\end{myproof}

Following the general scheme of a proof of the Berry-Esseen bound we are bound to work with the moments of a random vector, therefore in need of the technical lemma.
\begin{lemma}\label{Moments}
	Under the assumptions in the framework above holds
	\[
	\E_{i}\|\xiv_i\|^3\leq 14 \left(Tr\Sigma/n\right)^{3/2}
	\]
\end{lemma}
\begin{myproof}
	Generally for - $\E_{i}\|\xiv_i\|^k$ we can write
	\[
	\E_{i}\|\xiv_i\|^k=\frac{k}{2}\int_{0}^{\infty}\P(\|\xiv_i\|^2>s)s^{(k-2)/2}ds
	\]
	Theorem 3.1 or 4.1 from Spokoiny, Zhilova \cite{spokoiny2013sharp} on the sharp deviation bound for a sub-Gaussian vector $\xiv_i$ suggests for $s>0$
	\[
	\P(\|\xiv_i\|^2>p+s)\leq 2 e^{-\frac{ s^2}{6.6p}\vee \frac{s}{6.6}}.
	\]
	and additionally for a $s\in\left[-p,0\right]$ we know
	\[
	\P(\|\xiv_i\|^2>p+s)\leq 1.
	\]
	We also note that under dimension we understand $p=Tr\Sigma/n$ in the derivation. Since it was mentioned in the introduction that we work with an appropriately scaled random vectors. 
	
	Using following change of variables $s=s^{\prime}+\frac{p}{n}$ we come at the inequality
	\[
	\E_{i}\|\xiv_i\|^k=\frac{k}{2}\int_{-p}^{\infty}\P(\|\xiv_i\|^2>p+s^{\prime})(s^{\prime}+p)^{(k-2)/2}ds^{\prime}\leq
	\]
	\[
	\leq k\int_{0}^{\infty}e^{-\frac{ s^2}{6.6p}}(s^{\prime}+p)^{(k-2)/2}ds^{\prime}+\frac{k}{2}\int_{-p}^{0}(s^{\prime}+p)^{(k-2)/2}ds^{\prime}\leq
	\]
	\[
	\leq kp^{k/2}\int_{0}^{\infty}e^{-\frac{ ps^2}{6.6}}(s^{\prime}+p)^{(k-2)/2}ds^{\prime}+k\left(p\right)^{k/2}\leq
	\]
	\[
	\leq kp^{k/2}\int_{0}^{\infty}e^{-\frac{ s^2}{6.6}}(s+1)^{(k-2)/2}ds+kp^{k/2}
	\]
	and explicit calculation for $k=3$ yields the result
	\[
	\E_{i}\|\xiv_i\|^3\leq 14 p^{3/2}.
	\]
\end{myproof}

\subsection{Log-likelihood multiplier re-sampling}
\begin{theorem}
	The parametric model ($\ref{AMIV}$) in the introduction - $\delta_k=0$ -  under the assumption ($\ref{FSC}$) enables
	\[
	\left|\P\left(\left(T_{LR}-J\right)/\sqrt{J}>z^{\sbt}_{\alpha}\right)-\alpha\right|\leq C_0\frac{J^{3/2}}{\sqrt{Kn}}+C_1\sqrt{\frac{J\log J+x}{Kn}}
	\]
	with a dominating probability $>1-C_2e^{-x}$ and universal constants $C_0,C_1<\infty$.
\end{theorem}
\begin{proof}
	Using respective Wilks expansions let us reduce the log-likelihood ratio statistics $T_{LR}, T_{BLR}$ to the norms of score vectors $\|\xiv^s\|,\|\xiv_{\sbt}^{s}\|$ - sub-exponential random vectors based on the finite sample theory assumptions (section [$\ref{FST}$]). One has from the theorems [$\ref{SquareRootWilksResult}$,$\ref{BootstrapSquareRootWilksResult}$]
	\[
	\left|\sqrt{2T_{LR}}-\|\xiv^s\|\right|\leq C\left(J+x\right)/\sqrt{Kn},
	\]
	\[
	\left|\sqrt{2T_{BLR}}-\|\xiv^s_{\sbt}\|\right|\leq C\left(J+x\right)/\sqrt{Kn}.
	\]
	Both score vectors are reduced to the respective Gaussian counterparts 
	\[
	\widetilde{\xiv}^s_{\sbt}\sim\mathcal{N}\left(0,\frac{1}{n}\suml_i\xiv_i^s\xiv_i^{sT}\right),\;\;\widetilde{\xiv}^{s}\sim\mathcal{N}\left(0,\E\xiv^s\xiv^{sT}\right).
	\]
	In view of Gaussian approximation result (theorem $\ref{SGAR}$) one can state
	\[
	\sup_{t}\left|\P\left(\|\xiv^s\|<t\right)-\P\left(\|\widetilde{\xiv}^s\|<t\right)\right|\leq C\frac{J^{3/2}}{\sqrt{Kn}}
	\]
	with a universal constant $C<\infty$, and analogously the theorem implies
	\[
	\sup_{t}\left|\P\left(\|\xiv_{\sbt}^{s}\|<t\right)-\P\left(\|\widetilde{\xiv}_{\sbt}^{s}\|<t\right)\right|\leq C\frac{J^{3/2}}{\sqrt{Kn}}
	\]
	with the universal constant $C<\infty$. In turn, Gaussian comparison result [$\ref{GC}$] and Bernstein matrix inequality allow to derive
	\[
	\sup_{t}\left|\P\left(\|\widetilde{\xiv}^s_{\sbt}\|<t\right)-\P\left(\|\widetilde{\xiv}^s\|<t\right)\right|\leq C_1\sqrt{J}\|I-\left(\E\xiv^s\xiv^{sT}\right)^{-1/2}\frac{1}{n}\suml_i\xiv_i^s\xiv_i^{sT}\left(\E\xiv^s\xiv^{sT}\right)^{-1/2}\|_{op}
	\]
	\[
	\overset{\text{thm}\ref{MatrixBernstein2}}{\leq}C_1\sqrt{\frac{J\log J +x}{Kn}}
	\]
	with an exponentially large probability $1-C_2e^{-x}$.
	
	Finally, let us use the anti-concentration result (theorem 2.7) from {G{\"o}tze}, F. and {Naumov}, A. and {Spokoiny}, V. and {Ulyanov}, V. \cite{2017arXiv170808663G}, stating for a Gaussian vector $\xv\sim\mathcal{N}\left(0,\Sigma\right)$
	\[
	\P\left(t<\|\xv\|<t+\epsilon\right)\leq\frac{C\epsilon}{\|\Sigma\|_{Fr}}.
	\]
	It allows to translate Wilks expansions into the probabilistic language. Assembling all the statements in a cohesive structure one comes at
	\[
	\sup_{t}\left|\P\left(\left(T_{LR}-J\right)/\sqrt{J}<t\right)-\P\left(\left(T_{BLR}-J\right)/\sqrt{J}<t\right)\right|
	\]
	\[
	\overset{\text{Wilks+AC+GAR}}{\leq}\sup_{t}\left|\P\left(\left(\|\widetilde{\xiv}^s\|-J\right)/\sqrt{J}<t\right)-\P\left(\left(\|\widetilde{\xiv}^s_{\sbt}\|-J\right)/\sqrt{J}<t\right)\right|+2C\frac{J+x}{\sqrt{JKn}}+2C\frac{J^{3/2}}{\sqrt{Kn}}
	\]
	\[
	\overset{\text{GComp}}{\leq}C_1\sqrt{\frac{J\log J +x}{Kn}}+2C\frac{J+x}{\sqrt{JKn}}+2C\frac{J^{3/2}}{\sqrt{Kn}},
	\]
	which helps to infer straightforwardly the statement of the theorem.
\end{proof}

\bibliographystyle{plain}
\bibliography{Article}
\newpage
\begin{figure}[b]
	\caption{The empirical power of $T_{LR}$, $T_{BLR}$ and $T_{CLR}$ with weak instruments.} \label{Fig.PowerEnvelopeWeakLRBLRCLR}
	\includegraphics[scale=0.8,trim = 30mm 100mm 0mm 90mm]{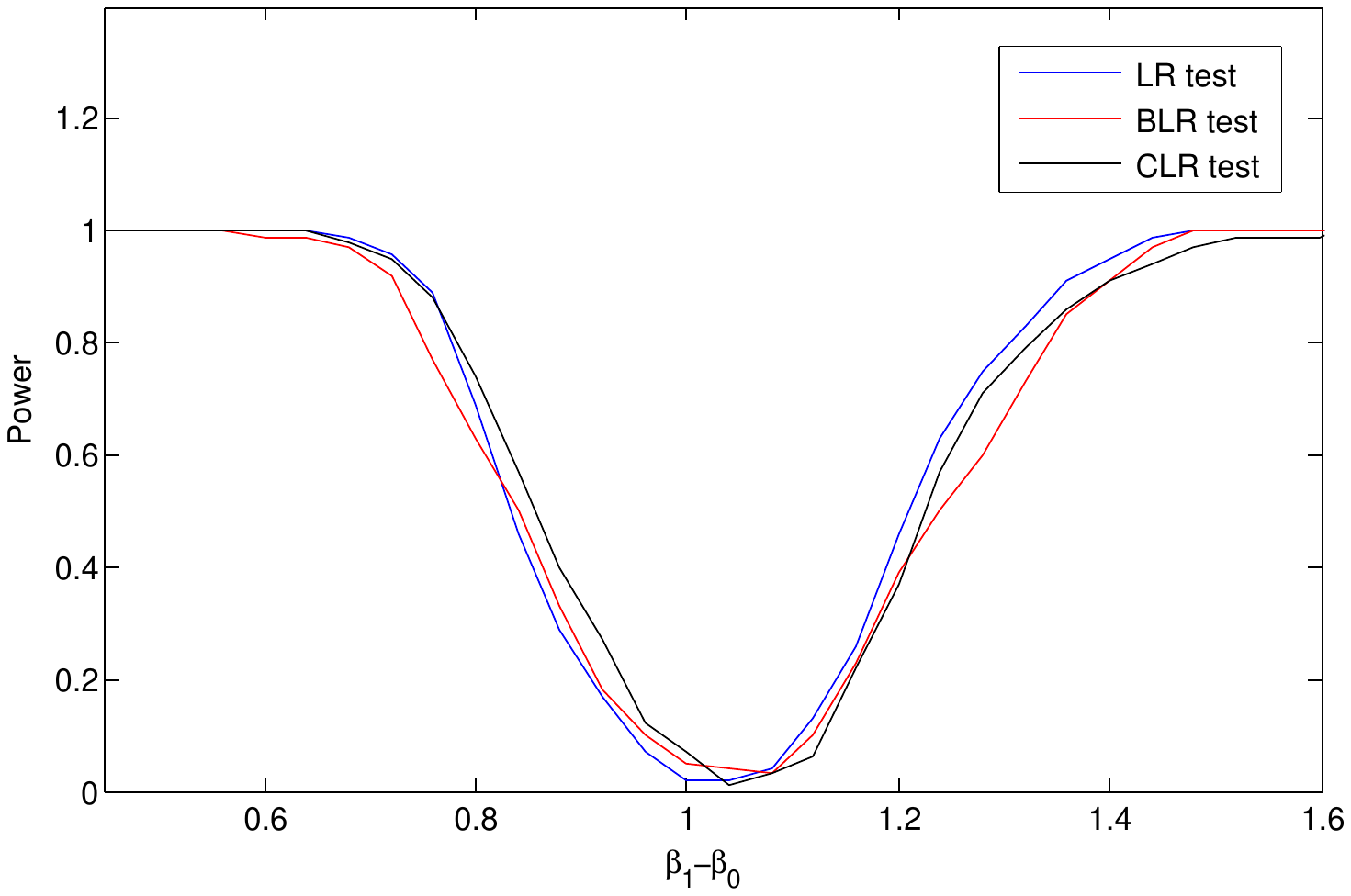}
	\bigskip
	DATA:
	n=200,\;q=5,\;$\piv^{*T}\Zv\Zv^T\piv^*=\frac{4}{n}$,\;$\Omega=\Matrix{1}{0}{0}{1}$
\end{figure}

\begin{figure}[b]
	\caption{The empirical power of $T_{LR}$, $T_{AR}$ and $T_{LM}$ with weak instruments.}\label{Fig.PowerEnvelopeWeakLRARLM}
	\includegraphics[scale=0.8,trim = 30mm 100mm 0mm 90mm]{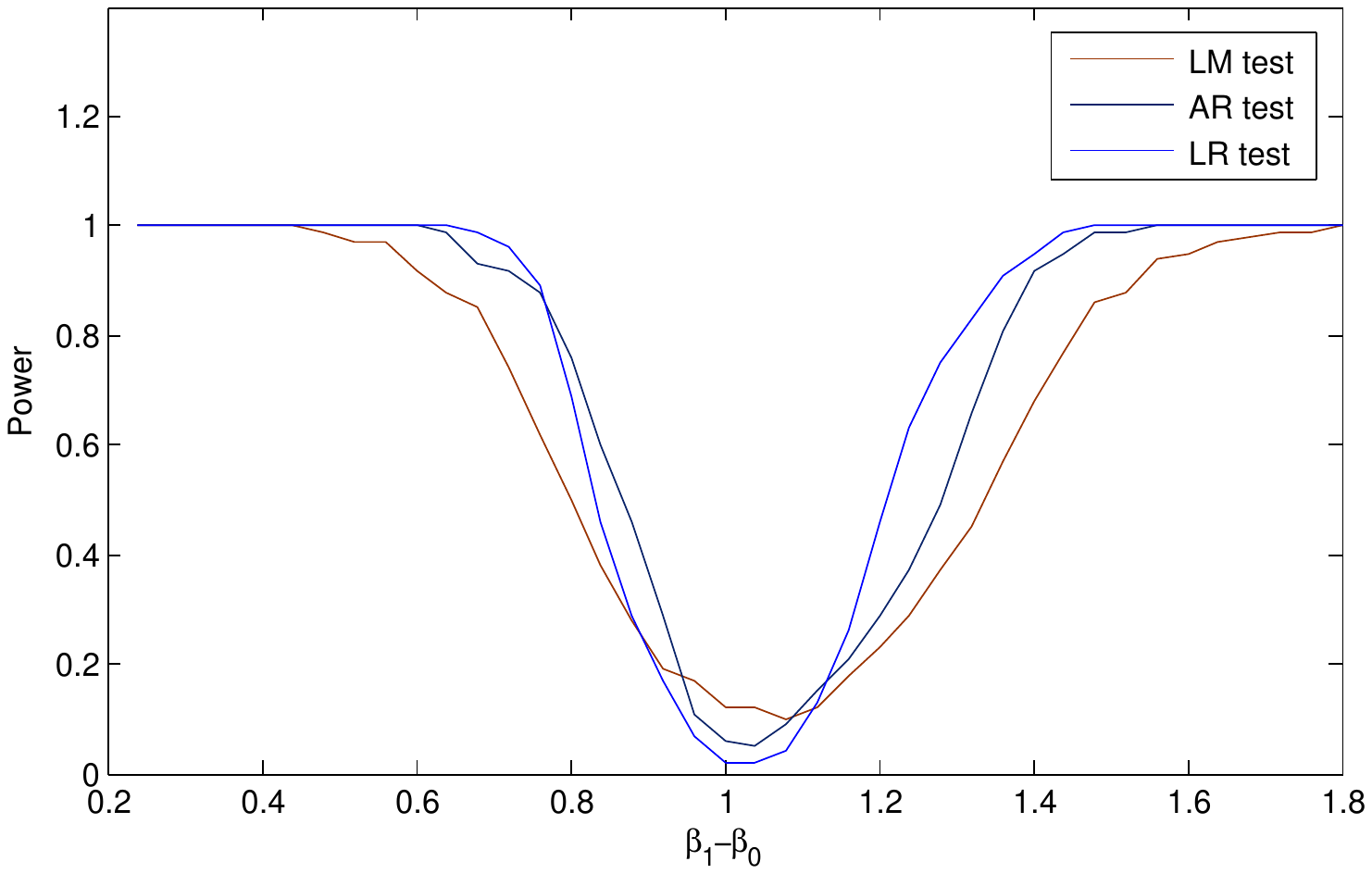}
	\bigskip
	DATA:
	n=200,\;q=5,\;$\piv^{*T}\Zv\Zv^T\piv^*=\frac{4}{n}$,\;$\Omega\Matrix{1}{0}{0}{1}$
\end{figure}

\begin{table}[b]
	\caption{\textbf{Power weak instrumental variables.}}
	\label{Tab:PowerEnvelopeWeakLRBLRCLRARLM}
	\begin{center}
		\begin{tabular}{|l|l|l|l|l|l|}
			\hline
			\textbf{$\beta_1-\beta_0$}&\textbf{LR}&\textbf{BLR}&\textbf{CLR}&\textbf{AR}&\textbf{LM}\\\hline
			0.48&1&1&1&0.99&1\\\hline
			0.56&1&1&1&0.97&1\\\hline
			0.64&1&0.99&1&0.88&0.99\\\hline
			0.72&0.96&0.92&0.95&0.74&0.92\\\hline
			0.8&0.69&0.63&0.74&0.5&0.76\\\hline
			0.88&0.29&0.33&0.4&0.28&0.46\\\hline
			0.96&0.07&0.1&0.12&0.17&0.11\\\hline
			1.04&0.02&0.04&0.01&0.12&0.05\\\hline
			1.12&0.13&0.1&0.06&0.12&0.15\\\hline
			1.2&0.46&0.39&0.37&0.23&0.29\\\hline
			1.28&0.75&0.6&0.71&0.37&0.49\\\hline
			1.36&0.91&0.85&0.86&0.57&0.81\\\hline
			1.44&0.99&0.97&0.94&0.77&0.95\\\hline
			1.52&1&1&0.99&0.88&0.99\\\hline
			1.6&1&1&0.99&0.95&1\\\hline
			1.68&1&1&1&0.98&1\\\hline
			1.76&1&1&1&0.99&1\\\hline
		\end{tabular}
	\end{center}
	\bigskip
	
	DATA:
	n=200,\;q=5,\;$\piv^{*T}\Zv\Zv^T\piv^*=\frac{4}{n}$,\;$\Omega\Matrix{1}{0}{0}{1}$
\end{table}

\begin{figure}[b]
	\caption{The empirical power of $T_{LR}$, $T_{BLR}$ and $T_{CLR}$ with weak instruments and Laplace errors.}\label{WeakMisspecifiedLaplace}
	\includegraphics[scale=0.7,trim = 30mm 100mm 0mm 90mm]{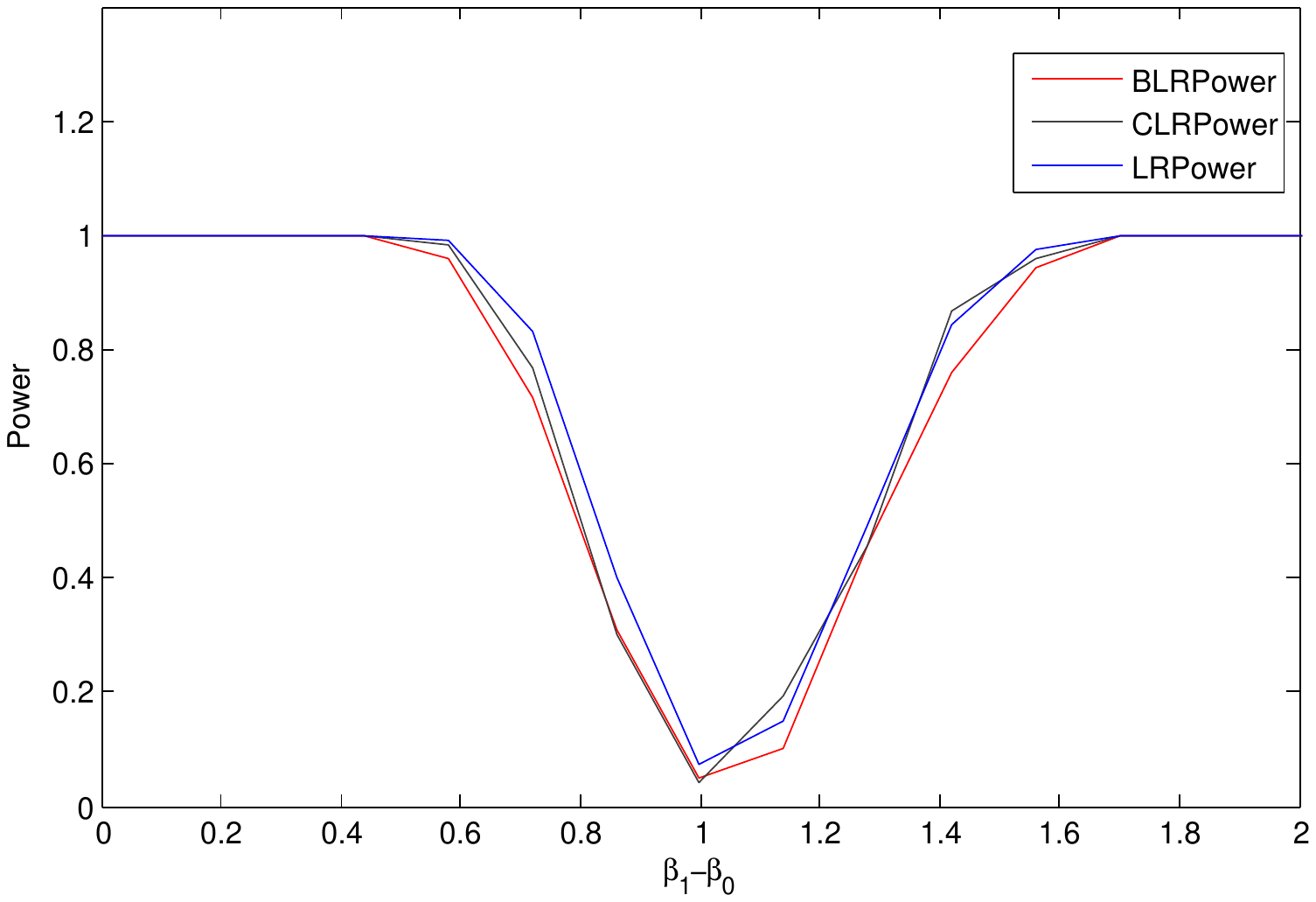}
	\bigskip
	DATA:\;n=200,\;q=5,\;$\piv^{*T}\Zv\Zv^T\piv^*=\frac{2.56}{n}$
\end{figure}

\begin{figure}[b]
	\caption{The empirical power of $T_{LR}$, $T_{AR}$ and $T_{LM}$ with weak instruments and Laplace errors.}\label{WeakMisspecifiedLaplaceARLM}
	\includegraphics[scale=0.7,trim = 30mm 100mm 0mm 90mm]{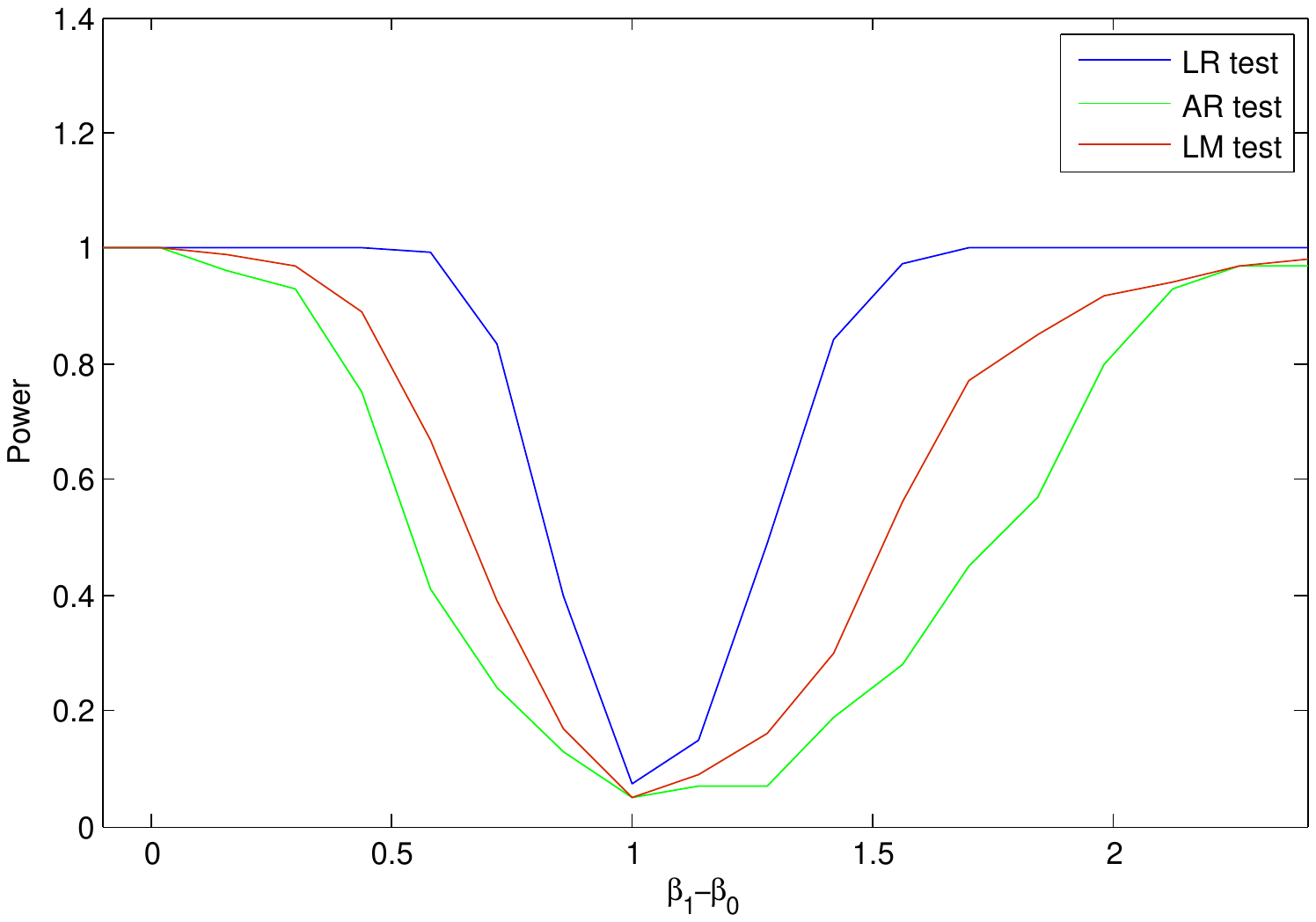}
	\bigskip
	DATA:\;n=200,\;q=5,\;$\piv^{*T}\Zv\Zv^T\piv^*=\frac{2.56}{n}$
\end{figure}

\begin{table}[b]
	\caption{\textbf{Power - Weak instrumental variables and Laplace noise.}} \label{Tab:PowerEnvelopeWeakMispecifiedLaplaceLRBLRCLRARLM}
	\begin{center}
		\begin{tabular}{|l|l|l|l|l|l|}
			\hline
			\textbf{$\beta_1-\beta_0$}&\textbf{LR}&\textbf{BLR}&\textbf{CLR}&\textbf{AR}&\textbf{LM}\\\hline
			0.02&1&1&1&1&1\\\hline
			0.16&1&1&1&0.96&0.99\\\hline
			0.3&1&1&1&0.93&0.97\\\hline
			0.44&1&1&1&0.75&0.89\\\hline
			0.58&0.99167&0.95833&0.98333&0.41&0.67\\\hline
			0.72&0.83333&0.71667&0.76667&0.24&0.39\\\hline
			0.86&0.4&0.30833&0.3&0.13&0.17\\\hline
			1&0.075&0.05&0.041667&0.05&0.05\\\hline
			1.14&0.15&0.1&0.19167&0.07&0.09\\\hline
			1.28&0.49167&0.45833&0.45833&0.07&0.16\\\hline
			1.42&0.84167&0.75833&0.86667&0.19&0.3\\\hline
			1.56&0.975&0.94167&0.95833&0.28&0.56\\\hline
			1.7&1&1&1&0.45&0.77\\\hline
			1.84&1&1&1&0.57&0.85\\\hline
			1.98&1&1&1&0.8&0.92\\\hline
			2.12&1&1&1&0.93&0.94\\\hline
			2.26&1&1&1&0.97&0.97\\\hline
		\end{tabular}
	\end{center}
	\bigskip
	DATA:\;n=200,\;q=5,\;$\piv^{*T}\Zv\Zv^T\piv^*=\frac{2.56}{n}$
\end{table}

\begin{figure}[b]
	\caption{The empirical power of $T_{LR}$, $T_{BLR}$ and $T_{CLR}$ with weak instruments and heteroskedastic errors.}\label{WeakMisspecifiedHeteroskedastic}
	\includegraphics[scale=0.7,trim = 30mm 100mm 0mm 90mm]{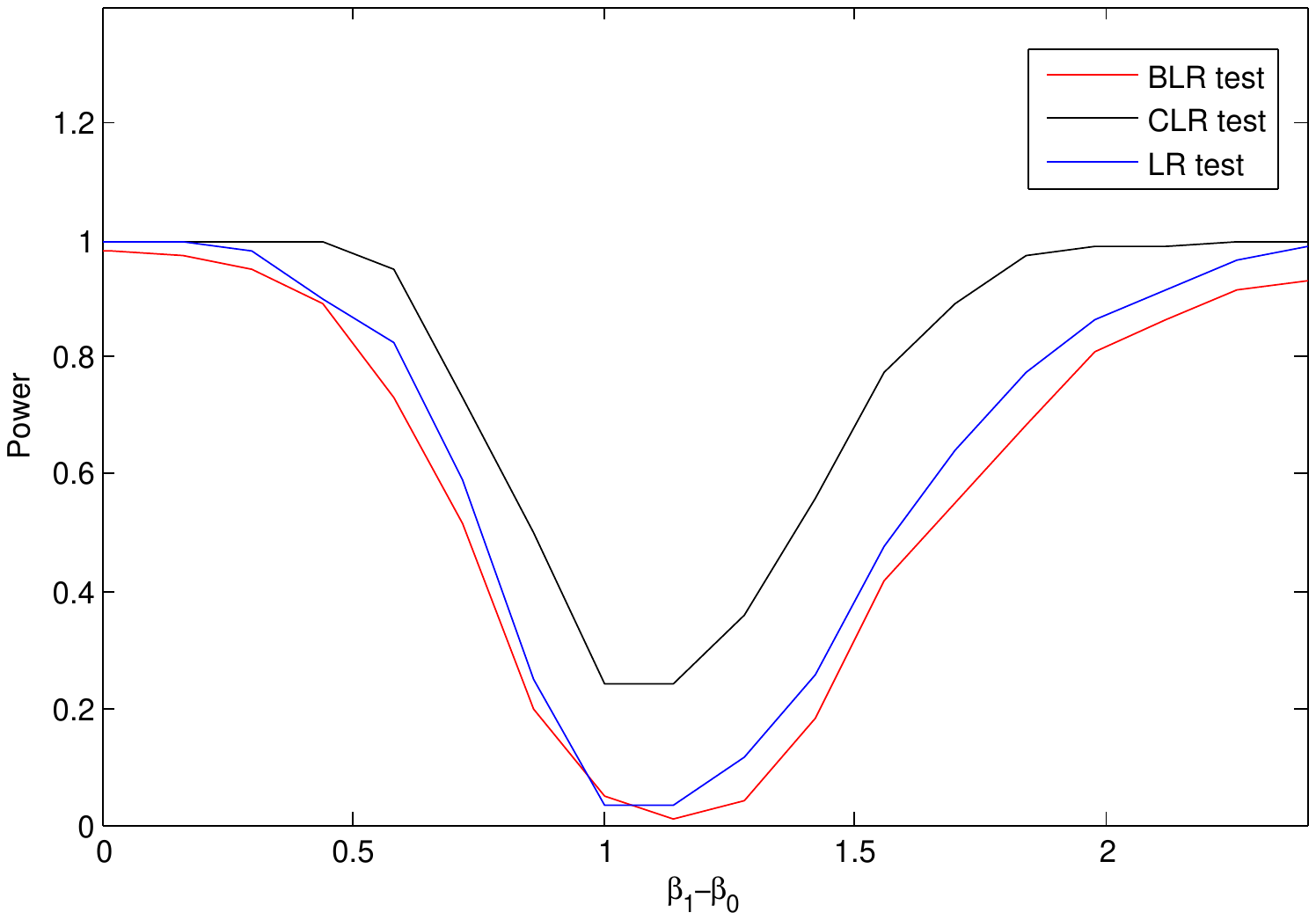}
	\bigskip
	DATA:\;n=200,\;q=5,\;$\piv^{*T}\Zv\Zv^T\piv^*=\frac{2.56}{n}$
\end{figure}

\begin{table}[b]
	\caption{\textbf{Power - Weak instrumental variables and heteroskedastic noise.}}\label{Tab:PowerEnvelopeWeakMispecifiedHeteroskedasticLRBLRCLRARLM}
	\begin{center}
		\begin{tabular}{|l|l|l|l|l|l|}
			\hline
			\textbf{$\beta_1-\beta_0$}&\textbf{LR}&\textbf{BLR}&\textbf{CLR}&\textbf{AR}&\textbf{LM}\\\hline
			-0.26&1&0.99167&1&1&1\\\hline
			-0.12&1&0.98333&1&1&1\\\hline
			0.02&1&0.98333&1&1&1\\\hline
			0.16&1&0.975&1&1&1\\\hline
			0.3&0.98333&0.95&1&1&1\\\hline
			0.44&0.9&0.89167&1&0.98333&0.99167\\\hline
			0.58&0.825&0.73333&0.95&0.94167&0.93333\\\hline
			0.72&0.59167&0.51667&0.73333&0.80833&0.8\\\hline
			0.86&0.25&0.2&0.5&0.7&0.50833\\\hline
			1&0.033333&0.05&0.24167&0.5&0.24167\\\hline
			1.14&0.033333&0.0083333&0.24167&0.50833&0.19167\\\hline
			1.28&0.11667&0.041667&0.35833&0.63333&0.35833\\\hline
			1.42&0.25833&0.18333&0.55833&0.73333&0.59167\\\hline
			1.56&0.475&0.41667&0.775&0.825&0.775\\\hline
			1.7&0.64167&0.55&0.89167&0.9&0.89167\\\hline
			1.84&0.775&0.68333&0.975&0.95833&0.95\\\hline
			1.98&0.86667&0.80833&0.99167&0.975&0.975\\\hline
			2.12&0.91667&0.86667&0.99167&1&0.99167\\\hline
			2.26&0.96667&0.91667&1&1&0.99167\\\hline
			2.4&0.99167&0.93333&1&1&0.99167\\\hline
		\end{tabular}
	\end{center}
	\bigskip
	DATA:\;n=200,\;q=5,\;$\piv^{*T}\Zv\Zv^T\piv^*=\frac{2.56}{n}$
\end{table}

\begin{figure}[b]
	\caption{The empirical power of $T_{LR}$, $T_{BLR}$ and $T_{CLR}$ with weak instruments and heteroskedastic (periodic) errors - case 3.} \label{WeakMisspecifiedHeteroskedasticPeriodic}
	\includegraphics[scale=0.7,trim = 30mm 100mm 0mm 90mm]{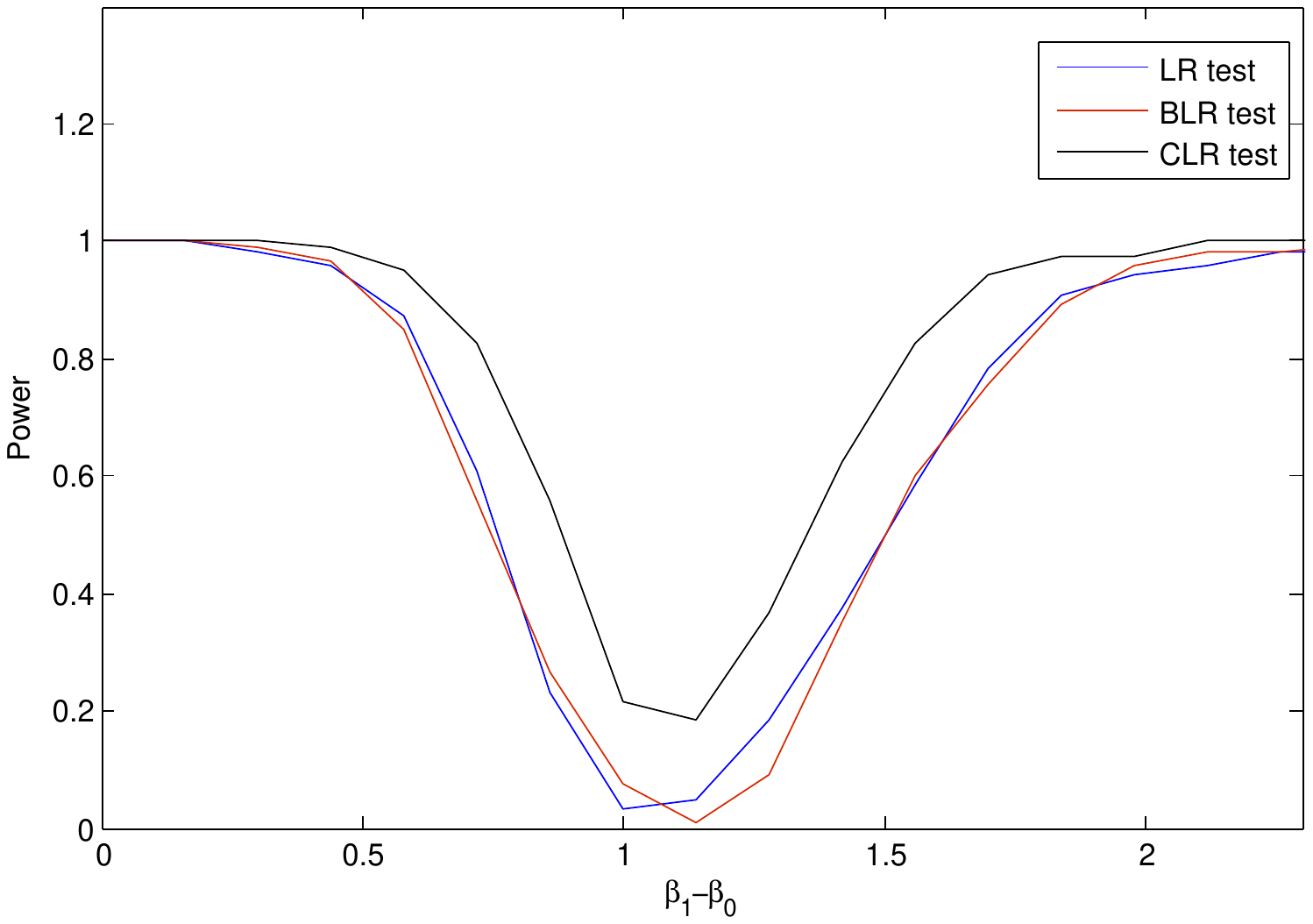}
	\bigskip
	DATA:\;n=200,\;q=5,\;$\piv^{*T}\Zv\Zv^T\piv^*=\frac{2.56}{n}$
\end{figure}

\begin{table}[b]
	\caption{\textbf{Power - Weak instrumental variables and Heteroskedastic periodic noise.}}
	\label{Tab:PowerEnvelopeWeakMispecifiedHeteroskedasticPeriodicLRBLRCLRARLM}
	\begin{center}
		\begin{tabular}{|l|l|l|l|l|l|}
			\hline
			\textbf{$\beta_1-\beta_0$}&\textbf{LR}&\textbf{BLR}&\textbf{CLR}&\textbf{AR}&\textbf{LM}\\\hline
			0.16&1&1&1&1&1\\\hline
			0.3&0.98333&0.99167&1&0.99167&1\\\hline
			0.44&0.95833&0.96667&0.99167&0.96667&1\\\hline
			0.58&0.875&0.85&0.95&0.9&0.98333\\\hline
			0.72&0.60833&0.55833&0.825&0.8&0.84167\\\hline
			0.86&0.23333&0.26667&0.55833&0.55&0.49167\\\hline
			1&0.033333&0.075&0.21667&0.35&0.175\\\hline
			1.14&0.05&0.0083333&0.18333&0.35&0.16667\\\hline
			1.28&0.18333&0.091667&0.36667&0.525&0.325\\\hline
			1.42&0.375&0.35&0.625&0.68333&0.58333\\\hline
			1.56&0.58333&0.6&0.825&0.81667&0.825\\\hline
			1.7&0.78333&0.75833&0.94167&0.91667&0.925\\\hline
			1.84&0.90833&0.89167&0.975&0.95833&0.96667\\\hline
			1.98&0.94167&0.95833&0.975&0.99167&0.99167\\\hline
			2.12&0.95833&0.98333&1&1&0.99167\\\hline
			2.26&0.98333&0.98333&1&1&0.99167\\\hline
			2.4&0.98333&0.99167&1&1&0.99167\\\hline
		\end{tabular}
	\end{center}
	\bigskip
	DATA:\;n=200,\;q=5,\;$\piv^{*T}\Zv\Zv^T\piv^*=\frac{2.56}{n}$
\end{table}
\end{document}